\documentclass[a4paper]{amsart}
\usepackage{soul} 
\usepackage[all]{xy}
\usepackage{amsmath}
\usepackage{amstext}
\usepackage{amssymb}
\usepackage{amscd}
\usepackage{soul}
\usepackage{enumerate}
\usepackage{graphicx}
\usepackage{tikz}
\usepackage{mathrsfs}
\usepackage{float}
\usepackage{rotating}
\usepackage{tikz-cd}
\usetikzlibrary{patterns}
\usepackage{color}
\usepackage{xcolor}
\usepackage[color,matrix,arrow]{xypic}
\usepackage{mathtools}
\xyoption{all}
\usepackage{pstricks}
\usepackage{graphicx}
\usepackage[bookmarksnumbered,colorlinks,plainpages]{hyperref}
\hypersetup{
	colorlinks=true,
	citecolor=blue
}
\newtheorem{theorem}{Theorem}[section]

\newtheorem{corollary}[theorem]{Corollary}
\newtheorem{lemma}[theorem]{Lemma}
\newtheorem{proposition}[theorem]{Proposition}
\newtheorem{definition-proposition}[theorem]{Definition-Proposition}

\theoremstyle{definition}
\newtheorem{definition}[theorem]{Definition}

\newtheorem{setting}[theorem]{Setting}
\newtheorem{remark}[theorem]{Remark}
\newtheorem{example}[theorem]{Example}

\numberwithin{equation}{section}

\renewcommand{\AA}{\mathcal{A}}

\newcommand{\CC}{\mathcal{C}}
\newcommand{\OO}{{\mathcal O}}
\newcommand{\UU}{\mathcal{U}}
\newcommand{\MM}{\mathcal{M}}

\renewcommand{\L}{\mathbb{L}}
\newcommand{\X}{\mathbb{X}}
\renewcommand{\P}{\mathbb{P}}
\newcommand{\Z}{\mathbb{Z}}
\newcommand{\mdots}{\begin{turn}{87}$\ddots$ \end{turn}}

\renewcommand{\c}{\vec c}

\newcommand{\y}{\vec y}
\newcommand{\x}{\vec x}

\newcommand{\z}{\vec z}

\newcommand{\w}{\vec{\omega}}

\newcommand{\cut}{\ar@{-}@[|(5)]}

\newcommand{\ca}{\operatorname{ca}\nolimits}

\newcommand{\Hom}{\operatorname{Hom}\nolimits}
\newcommand{\Dom}{\operatorname{Dom}\nolimits}

\newcommand{\End}{\operatorname{End}\nolimits}
\newcommand{\Ext}{\operatorname{Ext}\nolimits}

\renewcommand{\Im}{\operatorname{Im}\nolimits}
\newcommand{\coker}{\operatorname{coker}\nolimits}

\newcommand{\bo}{\operatorname{b}\nolimits}
\newcommand{\rk}{\operatorname{rk}\nolimits}
\newcommand{\gl}{\operatorname{gl}.\operatorname{dim}\nolimits}
\newcommand{\Lotimes}{\mathop{\stackrel{\mathbf{L}}{\otimes}}\nolimits}
\newcommand{\RHom}{\mathbf{R}\strut\kern-.2em\operatorname{Hom}\nolimits}
\newcommand{\Hhom}[3]{\mathcal{H}om_{#1}(#2,#3)}

\DeclareMathOperator{\add}{\mathsf{add}}

\DeclareMathOperator{\thick}{\mathsf{thick}}
\DeclareMathOperator{\CM}{\mathsf{CM}}
\DeclareMathOperator{\ACM}{\mathsf{ACM}}
\DeclareMathOperator{\moduleCategory}{\mathsf{mod}} \renewcommand{\mod}{\moduleCategory}
\DeclareMathOperator{\coh}{\mathsf{coh}}
\newcommand{\DDD}{\mathsf{D}}

\DeclareMathOperator{\lb}{\mathsf{line}}
\DeclareMathOperator{\vect}{\mathsf{vect}}
\DeclareMathOperator{\proj}{\mathsf{proj}}

\tikzset{every picture/.style={line width=0.75pt}} 

\begin{document}
	
	\title[ACM tilting bundles on a GL projective plane of type $(2,2,2,p)$]{ACM tilting bundles on a Geigle-Lenzing projective plane of type $(2,2,2,p)$}

\author[J. Chen,  S. Ruan and W. Weng] {Jianmin Chen,  Shiquan Ruan and Weikang Weng$^*$}

\thanks{$^*$ the corresponding author}
\makeatletter \@namedef{subjclassname@2020}{\textup{2020} Mathematics Subject Classification} \makeatother
	
	\subjclass[2020]{16G10, 14F08, 16G50, 18G80}
	\keywords{Geigle-Lenzing projective space, tilting bundle, arithmetically Cohen-Macaulay bundle, $2$-representation infinite algebra, weighted projective line}
	
	\begin{abstract} Let $\X$ be a Geigle-Lenzing projective plane of type $(2,2,2,p)$ and $\coh \X$  the category of coherent sheaves on $\X$. This paper is devoted to study ACM tilting bundles over $\X$, that is, tilting objects in the
	derived category $\DDD^{\bo}(\coh \X)$ that are also ACM bundles. We show that a tilting bundle consisting of line bundles is the $2$-canonical tilting bundle up to degree shift. We also provide a program to construct  ACM tilting bundles, which give a rich source of (almost) $2$-representation infinite algebras. As an application, we give a  classification result of  ACM tilting bundles. 
	\end{abstract}
	
	\maketitle
	
	\section{Introduction}	
	Geigle-Lenzing  projective spaces, introduced by Herschend, Iyama, Minamoto and Oppermann in \cite{HIMO}, are a higher dimensional generalization of Geigle-Lenzing weighted projective lines  \cite{GL}, and their representation theory was studied.	The study of Geigle-Lenzing projective spaces has a high contact with many mathematical subjects, including representation theory, singularity theory, (non-commutative) algebraic geometry and homological mirror symmetry, see e.g. \cite{CRW, FU, HIMO, IL, KLM, LO}. A different interpretation of the category of coherent sheaves on a Geigle-Lenzing  projective space is the module category of a Geigle-Lenzing order, see \cite{CI,RV} for the dimension one case and \cite{IL} for more general cases.

	The concept of tilting objects is essential for constructing derived equivalences. Tilting bundles and tilting complexes on weighted projective lines have been widely studied, see e.g. \cite{B,CLR1, CRZ,HR, J2, Me}. 
	The canonical tilting bundles \cite{GL} are a class of important tilting objects on weighted projective lines,  as their endomorphism algebras yield Ringel's canonical algebras \cite{Rin}. This result was later generalized to a higher dimensional version in \cite{HIMO}, where they constructed  a $d$-canonical tilting bundle on a Geigle-Lenzing  projective space, whose endomorphism algebra is a $d$-canonical algebra.  The squid tilting sheaves \cite{BKL} is also special on weighted projective lines.  Lerner and Oppermann in \cite{LO} presented a higher dimensional analog of them, whose endomorphism algebras are called $d$-squid algebras.

	The notion of $d$-representation infinite algebras was introduced in \cite{HIO}. This algebra is a  generalization of a hereditary representation infinite algebra, and  plays an important role in the context of higher Auslander-Reiten theory \cite{I1,I3}. However, fewer examples of $d$-representation infinite algebras have been calculated. At present, the main sources of such algebras arise from tensor product constructions \cite{HIO}, dimer models \cite{N}, skew-group algebras such as in  \cite{DG,Gi}, $d$-tilting bundle on Geigle-Lenzing  projective spaces \cite{HIMO} and (non-commutative) algebraic geometry \cite{M,MM} as well as via the action of higher APR-tilting \cite{MY}. Also, it is known that such an algebra can be obtained as the degree zero part of a bimodule $(d+1)$-Calabi-Yau algebra of Gorenstein parameter one \cite{AIR,Ke,MM}. 	
	Almost $d$-representation infinite algebras were introduced in \cite{HIMO}, which is a generalization of $d$-representation infinite algebras. A typical example is a $d$-canonical algebra as mentioned.

	Recently, we introduced the notion of $2$-extension bundles on a Geigle-Lenzing projective plane $\X$ with weight quadruple to study the category $\ACM \X$ of arithmetically Cohen-Macaulay (ACM) bundles and its stable category  $\underline{\ACM}\, \X$ in \cite{CRW}. Moreover, many properties of $2$-extension bundles have been studied. In particular, we proved that when $\X$ is of type $(2,2,2,p)$, each indecomposable ACM bundle is either a line bundle or a $2$-extension bundle. This motivates us to study  ACM tilting bundles on $\X$, that is,  tilting bundles that are also ACM bundles.  
	
	In this paper, we investigate ACM tilting bundles on $\X$ of type $(2,2,2,p)$, and  obtain a new family of (almost) $2$-representation infinite algebras. Our results can be viewed as a 
	higher dimensional generalization of the classification theorem on tilting bundle over a weighted projective line of type $(2,2,n)$, discussed in \cite{CLR1}.

	This paper is organized as follows. In Section \ref{sec:Preliminaries}, we recall some basic concepts and facts on  a   Geigle-Lenzing projective plane $\X$ of type $(2,2,2,p)$.  In Section \ref{sec:Indecomposable ACM bundles of rank four}, 
	we introduce the concepts of rigid domain and roof, together with their essential features, which is useful for  
	constructing ACM tilting bundles on $\X$. 	In Section \ref{sec:ACM tilting bundles on},  we  prove that a tilting bundle consisting of line bundles is the $2$-canonical tilting bundle  up to degree shift, see Theorem \ref{ACM tilting bundles consisting of line bundles}.  
	Furthermore, we provide a program to construct a family of ACM tilting bundles, which have the trichotomy form $T= (\bigoplus_{i\le n\le j}E_n) \oplus 	 V_{g,h}(E_i,E_j) \oplus S^{I}_{g,h}(E_i,E_j)$ for $0\le i\le j \le p-2$, see (\ref{form of tilting bundles}) for details. We show that such ACM tilting bundles are characterized by a certain condition among all ACM tilting bundles, see Theorem \ref{main theorem}.
	As an application, in Section \ref{sec:A classification theorem} we classify all ACM tilting bundles contained in the $2$-cluster tilting subcategory 
	of $\ACM \X$, arising from a $1$-tilting object in $\underline{\ACM}\, \X$. 
	In Section \ref{sec: Endomorphism algebras of}, we show that the endomorphism algebra of the tilting bundle $T$ is an  almost $2$-representation infinite algebra. In particular, if $i=0$ and $j=p-2$, then this gives arise to a  $2$-representation infinite algebra, see Theorem \ref{endomorphism algebra of T}. 
	
	\section{Preliminaries} \label{sec:Preliminaries}
	In this section, we recall some basic notions and facts appearing in this paper. We refer to \cite{I1,I3,J1} on higher homological algebras, and to \cite{HIMO,KLM} on Geigle-Lenzing hypersurfaces and projective spaces. Throughout this paper, we restrict our treatment to the case for dimension $d=2$, that is, Geigle-Lenzing  projective planes, and where the base field $\mathbf{k}$ is supposed to be algebraically closed. 
	
	\subsection{$2$-cluster tilting subcategories} Let us  recall the notion of functorially finite subcategory of an additive category.
	A full subcategory $\CC$ of an additive category $\AA$ is called a
	 \emph{contravariantly finite subcategory}  if for every object in $A\in\AA$, there exists a morphism
	$f:C\to A$ with $C \in \CC$ such that the induced morphism
	$$\Hom_{\AA}(C',C)\to\Hom_{\AA}(C',A)$$ is surjective for all $C'\in\CC$. Such a morphism $f$ is called a \emph{right $\CC$-approximation} of $A$. Dually we define  a \emph{left $\CC$-approximation} and a 
	\emph{covariantly finite subcategory}. We say that $\CC$ is \emph{functorially finite} if it is both contravariantly and covariantly finite.

	\begin{definition}(\cite[definition 4.13]{J1}). A full subcategory $\CC$  of an exact category  $\AA$ is called  ($2$-)\emph{cluster tilting subcategory} if the following conditions are satisfied.	
		\begin{itemize}
			\item [(a)] $\CC$ is a functorially finite subcategory of $\AA$ such that	
			\begin{eqnarray*}
				\CC=\{X\in\AA\mid \Ext_{\AA}^1(\CC,X)=0\}=\{X\in\AA\mid \Ext_{\AA}^1(X,\CC)=0\}.
			\end{eqnarray*}
			\item [(b)]  $\CC$ is a both generating and cogenerating subcategory of $\AA$, that is, for any object $A\in \AA$ there exists an epimorphism $C \to A$ with $C \in \CC$ and a  monomorphism  $ A \to C'$ with  $C' \in \CC$.
		\end{itemize} 
	\end{definition} 

	\subsection{Geigle-Lenzing hypersurfaces  of type $(2,2,2,p)$}  
	 Fix a quadruple $(2,2,2,p)$ with an integer $p\ge 2$, called \emph{weights}.  Consider the  $\mathbf{k}$-algebra
	$$R=\mathbf{k}[X_1, \dots, X_4]/(X_1^{2}+X_2^{2}+X_3^{2}+X_4^{p}).$$ Let $\L$ be the abelian group on generators $\x_1,\ldots,\x_4,\c$, subject to the relations $$2 \x_1= 2 \x_2=2 \x_3=p \x_4=:\c.$$ The element $\c$ is called the \emph{canonical element} of $\L$. Each element $\x$ in $\L$ can be uniquely written in a \emph{normal form} as 
	\begin{align}\label{equ:nor}
		\x=\sum_{i=1}^4\lambda_i\x_i+\lambda\c
	\end{align}
	with integers $ \lambda_1, \lambda_2, \lambda_3 \in \{0,1\}$, $0\le \lambda_4 < p $ and $\lambda\in\Z$. We can regard $R$ as an $\L$-graded $k$-algebra by setting $\deg X_i:=\x_i$ for any $i$, hence $R=\bigoplus_{\x\in \L } R_{\x}$, where $R_{\x}$ denotes the homogeneous component of degree $\vec{x}$. The pair $(R,\L)$ is called a $\emph{Geigle-Lenzing \emph{(}GL\emph{)} hypersurface}$ of type $(2,2,2,p)$.

	Let $\L_+$ be the submonoid of $\L$ generated by $\x_1,\ldots,\x_4,\c$. Then we equip $\L$ with the structure of a partially ordered set: $\x\le \y$ if and only if $\y-\x\in\L_+$. Each element $\x$ of $\L$ satisfies exactly one of the following two possibilities
	\begin{align}\label{two possibilities of x}
		\x\geq 0 \text{\quad or\quad}  \x\leq \vec{\omega}+2\c,
	\end{align}
	where  $\w:=\c-\sum_{i=1}^4 \x_i$ is called the \emph{dualizing element} of $\L$. We have $R_{\x} \neq 0$ if and only if $\x \geq 0$ if and only if $\ell \geq 0$ in its normal form in (\ref{equ:nor}). We say that a subset $I$ of
	$\L$ is $\emph{convex}$ if for any $\x,\y,\z\in \L$ such that $\x\le \y\le\z$ and $\x,\z \in I$, we have $\y \in I$. The interval in $\L$ is denoted by
	$[\x,\y] := \{\z \in\L \mid \x \le \z \le \y\}.$ 
	
	We denote by $\CM^{\L} R$ the category of $\L$-graded (maximal) Cohen-Macaulay $R$-modules, which is Frobenius, and its stable category $\underline{\CM}^{\L}R$ forms a triangulated category 
	by a general result of Happel \cite{Hap1}. 
	The following presents a basic property in Cohen-Macaulay representation theory. Here, we denote by $D$ the $\mathbf{k}$-dual, that is  $D(-):=\Hom_{\mathbf{k}}(-,\mathbf{k})$.
	\begin{theorem}[\cite{HIMO}] (Auslander-Reiten-Serre duality) There exists a functorial isomorphism for any $X,Y\in \underline{\CM}^{\L}R$:
		\begin{equation}\label{Auslander-Reiten-Serre duality CM}
			\Hom_{\underline{\CM}^{\L}R}(X,Y)\simeq D\Hom_{\underline{\CM}^{\L}R}(Y,X(\w)[2]).
		\end{equation}	
	\end{theorem}
	
	\subsection{Geigle-Lenzing projective planes of type $(2,2,2,p)$}
	The category of $\emph{coherent sheaves on Geigle-Lenzing (GL) projective plane}$ ${\X}$ is defined as the quotient category
	\[ \coh \X :=\mod^{\L}R/\mod_0^{\L}R\]
	of the category of finitely generated $\L$-graded  $R$-modules $\mod^{\L}R$ modulo its Serre subcategory $\mod^{\L}_0R$ consisting of finite dimensional $R$-modules. 
	
	Denote by $\pi:\mod^{\L}R\to\coh\X$ the natural functor, also called \emph{sheafification}. The object $\OO:=\pi(R)$ is called the \emph{structure sheaf} of $\X$. We recall a list of fundamental properties of  $\coh\X$ and its derived category $\DDD^{\bo}(\coh\X)$.
	
	\begin{theorem}[\cite{HIMO}] \label{basic properties} The category $\coh\X$ has the following properties:
		\begin{itemize}
			\item[(a)] $\coh\X$ is a Noetherian abelian category of global dimension $2$.
			\item[(b)] $\Hom_{\DDD^{\bo}(\coh\X)}(X,Y)$ is
			a finite dimensional $\mathbf{k}$-vector space for all $X,Y\in\DDD^{\bo}(\coh\X)$. In particular, $\coh\X$ is Ext-finite.
			\item[(c)] For any $\x,\y \in \L$ and $i\in\Z$,  
			\begin{align}\label{extension spaces between line bundles}
				\Ext_{\X}^i(\OO(\x),\OO(\y))=\left\{
				\begin{array}{ll}
					R_{\y-\x}&\mbox{if $i=0$,}\\
					D(R_{\x-\y+\w})&\mbox{if $i=2$,}\\
					0&\mbox{otherwise.}
				\end{array}\right.
			\end{align}
			\item[(d)] (Auslander-Reiten-Serre duality) There exists a functorial isomorphism for any $X, Y\in\DDD^{\bo}(\coh\X) \mathsf{:}$
			\begin{equation}\label{Auslander-Reiten-Serre duality}
				\Hom_{\DDD^{\bo}(\coh\X)}(X,Y)\simeq D\Hom_{\DDD^{\bo}(\coh\X)}(Y,X(\w)[2]).
			\end{equation}
			In other words, $\DDD^{\bo}(\coh\X)$ has a Serre functor $(\w)[2]$.	
		\end{itemize}
	\end{theorem}
	
		Denote by $\vect \X$ the full subcategory of $\coh \X$ formed by all vector bundles 
	and by $\lb\X$ its full subcategory  of finite direct sums of 
	line bundles $\OO(\x)$ with $\x\in \L$. The sheafification  $\pi:\mod^{\L}R\to\coh\X$ restricts to a fully faithful 
	functor $\CM^{\L}R\to\vect\X$ and further induces an equivalence $\CM^{\L}R \xrightarrow{\sim} \pi(\CM^{\L} R)$.  In the context of projective geometry (e.g. \cite{CH,CMP}), the objects in $\pi(\CM^{\L} R)$ are called \emph{arithmetically Cohen-Macaulay} (\emph{ACM}) \emph{bundles}. Denote by $\ACM \X:=\pi(\CM^{\L} R)$. The following gives a description of $\ACM \X$ in terms of vector bundles by \cite{HIMO}.
	\begin{align}\label{ACM bundle eq.}
		\ACM \X&=
		\{X\in\vect\X\mid \Ext_{\X}^1(\OO(\x),X)=0 \text{ for any} \ \x\in\L\}\\
		&=\{X\in\vect\X\mid \Ext_{\X}^1(X,\OO(\x))=0\text{ for any} \ \x\in\L\}\nonumber. 
	\end{align}			
	Notice that all line bundles serve as the indecomposable projective-injective objects in $\ACM \X$ by (\ref{ACM bundle eq.}). Moreover, $\pi$ restricts to an equivalence $\proj^{\L}R\xrightarrow{\sim}\lb\X$ with the $\L$-graded  indecomposable projective module $R(\x)$ corresponding to the line bundle $\OO(\x)$. Therefore, this turns $\ACM \X$ into a Frobenius category and thus its stable category $\underline{\ACM} \, \X$ is triangle equivalent to $\underline{\CM}^{\L}R$. 
	
	For convenient, we denote ${\Ext}^{i}_{\coh \X}(-,-)$ by ${\Ext}^{i}(-,-)$ for $i\ge 0$, and denote ${\Hom}_{\underline{\ACM}\, \X}(-,-)$ by  $\underline{\Hom}(-,-)$.
	\subsection{$2$-tilting bundles and Grothendieck groups} Let $\mathcal{T}$ be a triangulated category with a suspension functor $[1]$. 
	
	\begin{definition} A basic object $V \in \mathcal{T}$ is called \emph{tilting} in $\mathcal{T}$ if
		\begin{itemize}
			\item[(a)] $V$ is \emph{rigid}, that is $\Hom_{\mathcal{T}}(V,V[i])=0$ for all $i \neq 0$.
			\item[(b)] $\thick V=\mathcal{T}$, where $\thick V$ denotes by the smallest thick subcategory of $\mathcal{T}$ containing $V$.		
		\end{itemize}	
	 If moreover $\End_{\mathcal{T}} (V)$ has global dimension at most $2$,  such a $V$ is called \emph{$2$-tilting} in $\mathcal{T}$.
	 A tilting object $V$ in  $\DDD^{\bo}(\coh \X)$ is called a \emph{tilting bundle} (resp.  \emph{ACM tilting bundle}) on $\X$ if $V \in \vect \X$ (resp. $\ACM \X$).
	\end{definition}	
	
	\begin{definition}(\cite[Definition 7.16]{HIMO})\label{def.slice}  Let $\UU$ be a $2$-cluster tilting subcategory of $\vect \X$. We call an object $V \in \UU$ $\emph{slice}$ in $\UU$ if the following conditions are satisfied.
		\begin{itemize}
			\item[(a)] ${\UU}=\add\{ {V} (\ell\w)\mid\ell\in\Z\}$.
			\item[(b)] $\Hom_{}({V},{V}(\ell\w))=0$ for any $\ell>0$.
		\end{itemize}
		In this case, $\add V$ meets each $\w$-orbit of indecomposable objects in $\UU$ exactly once.		
	\end{definition}

	It follows from a general result in \cite{HIMO} that a $2$-tilting bundle $V$ on $\X$ induces the $2$-cluster tilting subcategory ${\UU}:= \add \{V(\ell\w) \mid \ell \in \Z \}$ of $\vect \X$. Moreover,	the following theorem gives a more explicit expression between  $2$-tilting bundles and $2$-cluster tilting subcategories, which will be used frequently later.
	
	\begin{theorem}\emph{(}\cite[Theorem 7.17]{HIMO}\emph{)} \label{tilting-cluster tilting}
		Let $\X$ be a GL projective plane. Then $2$-tilting bundles on $\X$ are precisely
		slices in $2$-cluster tilting subcategories of $\vect\X$.
	\end{theorem}
	
	Let $K_0(\coh \X)$ be the Grothendieck group of $\coh \X$. We denote by $[X]$ the corresponding element in $K_0(\coh \X)$. It was shown in \cite{HIMO} that $K_0(\coh \X)$ is a free abelian group  with a  basis $[\OO(\x)]$ for $0 \le \x \le 2\c$. Further, the rank of $K_0(\coh \X)$ equals $5p+7$. The \emph{rank functor} {\rm rk}: $K_0(\coh \X) \to \Z$ is the additive function uniquely determined by $${\rm rank}\,([\OO(\x)])=1\  \text{for any}\  \x\in \L.$$ 	
	
	Recall that there are no nonzero morphisms from torsion sheaves (that is, rank zero objects in $\coh \X$) 
	to vector bundles by a general result of \cite[Proposition 3.28]{HIMO}.

		\subsection{Indecomposable rank-four ACM bundles} Let us recall the notions of $2$-(co)extension bundles and $2$-(co)Auslander bundles \cite{CRW}.	Let $L$ be a line bundle and $\x:=\lambda\x_4$ with an integer $0\le \lambda \le p-2$. We set
	\begin{align*}
		L\langle \x \rangle&:=\big(\bigoplus_{1\le i \le 3}L(\x-\x_i) \big) \oplus  L(-\x_4), \\ 
		L[\x]&:=\big(\bigoplus_{1\le i \le 3}L(\w+\x_i) \big)  \oplus  L(\w+\x+\x_4).
	\end{align*}
	Consider  short exact sequences
	\begin{align*}
		\eta_{1}: \ 0 \to K \xrightarrow{} L\langle \x \rangle \xrightarrow{\gamma} L(\x) \to 0 \ \text{ and } \
		\mu_{1}: \ 0 \to L(\w) \xrightarrow{\gamma'} L[\x] \to C \to 0
	\end{align*}
	in $\coh \X$, where  the morphisms $\gamma=\gamma':=((X_i)_{i=1}^{3}, X_4^{\lambda+1})$.  
	Applying $\Hom(-,L(\w))$ (resp. $\Hom(L(\x),-)$) to the  sequence $\eta_{1}$ (resp. $\mu_{1}$), one can check that
	$$\Ext^{1}(K,L(\w)) =\mathbf{k}\ \ (\text{resp.} \ \Ext^{1}(L(\x),C)=\mathbf{k}).$$ Let  
	$\eta_{2}:0\to L(\w) \xrightarrow{\alpha} E \xrightarrow{} K \to 0  \text{ (resp.  }  \mu_2: 0\to C \to F \xrightarrow{\alpha'} L(\x)  \to 0)$
	be a nonsplit exact sequence in $\coh \X$. By connecting these two sequences $\eta_{1}$ and $\eta_{2}$ (resp. $\mu_{1}$ and $\mu_{2}$), we finally obtain an exact sequence 
	\begin{align}
		\eta:& \quad	0 \to L(\w) \xrightarrow{\alpha}  E 
		\xrightarrow{\beta} 
		L\langle \x \rangle \xrightarrow{\gamma} L(\x) \to 0 \label{2-extension bundle exact} \\ 
		(\text{resp.} \ \mu:& \quad	0 \to L(\w)\xrightarrow{\gamma'} L[\x] \xrightarrow{\beta'} F \xrightarrow{\alpha'} L(\x) \to 0)
		\label{2-coextension bundle exact} 
	\end{align}
	in $\coh \X$. In this case,
	the term $E_L \langle \x \rangle:=E$ (resp. $F_L \langle \x \rangle:=F$)  of the sequence, which is uniquely determined up to isomorphism, is called the  \emph{$2$-extension bundle} (resp. \emph{$2$-coextension bundle})
	given by the data $(L,\x)$. If $L=\OO$, then we just write $E\langle \x \rangle$ (resp. $F\langle \x \rangle$).

	For $\x=0$, the morphism $ L(\w) \xrightarrow{\alpha}  E_L\langle 0 \rangle$ (resp. $F_L\langle 0 \rangle \xrightarrow{\alpha'} L(\x)$) is a left (resp. right) minimal almost split morphism in $\ACM \X$, and the
	term $E_L:=E_L\langle 0 \rangle$ (resp. $F_L:=F_L\langle 0 \rangle$) is called the 
	\emph{$2$-Auslander bundle} (resp. \emph{$2$-coAuslander bundle}) associated with $L$. 						 		
	We recall the following properties of  $2$-extension bundles. 
	
	\begin{proposition}[\cite{CRW}] \label{L-action on 2-extension bundles}
		Let  $0\le \x \le (p-2)\x_4$. The following assertions hold.
		\begin{itemize}
			\item[(a)]  We have $ F_{L}\langle \x \rangle\simeq E_{L}\langle \x \rangle[1]$.
			\item[(b)]   For any $\y\in\L$, we have $(E_{L}\langle \x \rangle)(\y)\simeq E_{L(\y)}\langle \x \rangle$.				
			\item[(c)] 	For any $1\le i \le 3$, there are isomorphisms
			$$  E_{L}\langle \x \rangle[1]\simeq E_{L}\langle \x \rangle(\x_i) \simeq E_{L}\langle (p-2)\x_4-\x \rangle(\x+\x_4).$$ 
		\end{itemize}
		Here, $[1]$ denotes the suspension functor in $\underline{\ACM} \, \X$.
	\end{proposition}

	According to \cite[Theorem 6.1]{CRW} (see also \cite[Theorem 4.76]{HIMO}),  $T_{\rm cub}:=\bigoplus_{ \ell=0}^{p-2} E\langle \ell\x_4 \rangle$ is a tilting object in $\underline{\ACM}\, \X$, whose endomorphism algebra  $\underline{\End}(T_{\rm cub})^{}\simeq  \mathbf{k}\vec{\mathbb A}_{p-1}$, where  $\vec{\mathbb A}_n$ denotes the equioriented quiver of type ${\mathbb A}_n$. Further,  it induces a  triangle equivalence 
	\[ \underline{\ACM} \, \X   \simeq \DDD^{\bo}(\mod \mathbf{k} \vec{\mathbb A}_{p-1}).\]	
	We identify objects in $\underline{\ACM} \, \X$ with objects in $\ACM \X$ without line bundles direct summands. 
	By adding all line bundles $\OO(\x)$ for $\x \in \L$  together with the arrows connected to $2$-(co)Auslander bundles into the Auslander-Reiten quiver  ${\mathfrak A}(\underline{\ACM} \, \X)$, we get the Auslander-Reiten quiver ${\mathfrak A}({\ACM} \, \X)$, see also \cite[Theorems 4.13, 4.46]{HIMO}.
	\begin{figure}[H]
	 \resizebox{\textwidth}{!}{
		\begin{xy} 0;<16pt,0pt>:<0pt,16pt>::
			(16,8) *+{\cdot} ="42",
			(18,10) *+{\bullet} ="52",
			(16,4) *+{} ="23",
			(16.4,4.5) *+{\mdots} ="x",
			(17.2,5.4) *+{\mdots} ="y",
			(18,6) *+{} ="33",
			(20,8) *+{\cdot} ="43",
			(22,10) *+{\bullet} ="53",
			(16,0) *+{\bullet} ="04",
			(18,2) *+{\cdot} ="14",
			(20,4) *+{} ="24",
			(20.4,4.5) *+{\mdots} ="x",
			(21.2,5.4) *+{\mdots} ="y",
			(22,6) *+{} ="34",
			(24,8) *+{\cdot} ="44",
			(26,10) *+{\bullet} ="54",
			(20,0) *+{\bullet} ="05",
			(22,2) *+{\cdot} ="15",
			(24,4) *+{} ="25",
			(24.4,4.5) *+{\mdots} ="x",
			(25.2,5.4) *+{\mdots} ="y",
			(26,6) *+{} ="35",
			(28,8) *+{\cdot} ="45",
			(30,10) *+{\bullet} ="55",
			(24,0) *+{(\w)} ="06",
			(26,2) *+{\langle 0,0 \rangle} ="16",
			(28,4) *+{} ="26",
			(28.4,4.5) *+{\mdots} ="x",
			(29.2,5.4) *+{\mdots} ="y",
			(30,6) *+{} ="36",
			(32,8) *+{\langle p-2, 0 \rangle} ="46",
			(34,10) *+{(-\x_4)} ="56",
			(28,0) *+{(\w+\x_4)} ="07",
			(30,2) *+{\langle 0,1 \rangle} ="17",
			(32,4) *+{} ="27",
			(32.4,4.5) *+{\mdots} ="x",
			(33.2,5.4) *+{\mdots} ="y",
			(34,6) *+{} ="37",
			(36,8) *+{\langle p-2, 1 \rangle} ="47",
			(38,10) *+{(0)} ="57",
			(32,0) *+{\bullet} ="08",
			(34,2) *+{\cdot} ="18",
			(36,4) *+{} ="28",
			(36.4,4.5) *+{\mdots} ="x",
			(37.2,5.4) *+{\mdots} ="y",
			(38,6) *+{} ="38",
			(40,8) *+{\cdot} ="48",
			(42,10) *+{\bullet} ="58",
			(36,0) *+{\bullet} ="09",
			(38,2) *+{\cdot} ="19",
			(40,4) *+{} ="29",
			(40.4,4.5) *+{\mdots} ="x",
			(41.2,5.4) *+{\mdots} ="y",
			(42,6) *+{} ="39",
			(44,8) *+{\cdot} ="49",
			(46,10) *+{\bullet} ="59",
			(40,0) *+{\bullet} ="010",
			(42,2) *+{\cdot} ="110",
			(44,4) *+{} ="210",
			(44.4,4.5) *+{\mdots} ="x",
			(45.2,5.4) *+{\mdots} ="y",
			(46,6) *+{} ="310",
			(44,0) *+{\bullet} ="011",
			(46,2) *+{\cdot} ="111",
			"42", {\ar"52"},
			"42", {\ar"33"},
			"52", {\ar"43"},
			"33", {\ar"43"},
			"43", {\ar"53"},
			"23", {\ar"14"},
			"43", {\ar"34"},
			"53", {\ar"44"},
			"04", {\ar"14"},
			"14", {\ar"24"},
			"34", {\ar"44"},
			"44", {\ar"54"},
			"14", {\ar"05"},
			"24", {\ar"15"},
			"44", {\ar"35"},
			"54", {\ar"45"},
			"05", {\ar"15"},
			"15", {\ar"25"},
			"35", {\ar"45"},
			"45", {\ar"55"},
			"15", {\ar"06"},
			"25", {\ar"16"},
			"45", {\ar"36"},
			"55", {\ar"46"},
			"06", {\ar"16"},
			"16", {\ar"26"},
			"36", {\ar"46"},
			"46", {\ar"56"},
			"16", {\ar"07"},
			"26", {\ar"17"},
			"46", {\ar"37"},
			"56", {\ar"47"},
			"07", {\ar"17"},
			"17", {\ar"27"},
			"37", {\ar"47"},
			"47", {\ar"57"},
			"17", {\ar"08"},
			"27", {\ar"18"},
			"47", {\ar"38"},
			"57", {\ar"48"},
			"08", {\ar"18"},
			"18", {\ar"28"},
			"38", {\ar"48"},
			"48", {\ar"58"},
			"18", {\ar"09"},
			"28", {\ar"19"},
			"48", {\ar"39"},
			"58", {\ar"49"},
			"09", {\ar"19"},
			"19", {\ar"29"},
			"39", {\ar"49"},
			"49", {\ar"59"},
			"19", {\ar"010"},
			"29", {\ar"110"},
			"49", {\ar"310"},
			"010", {\ar"110"},
			"110", {\ar"210"},
			"110", {\ar"011"},
			"210", {\ar"111"},
			%
			"011", {\ar"111"},
		\end{xy}
	} 
	\caption{The Auslander-Reiten quiver of $\ACM \X$} \label{ARq of X}
	\end{figure}
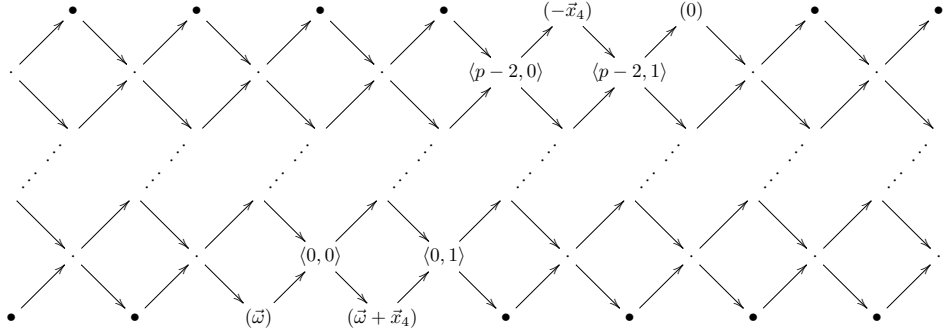
	Here, a vertex labeled $\langle i,j \rangle$ denotes the position of $2$-extension bundle  $E\langle i\x_4 \rangle(j\x_4)$ for $0\le i \le p-2$ and $j\in\Z$.
	The position of $\OO(\x)$ is denoted $(\x)$ or $\bullet$ in case $\x$ is not specified. The  
	positions of $\OO(\x)$ and $\OO(\x+\bar{x}_{ab})$ are overlapping, where $\bar{x}_{ab}:=\c-\x_a-\x_b$ for  $1\le a<b\le 3$. The interpretation is that every arrow starting or ending at  vertex $(\x)$ or $\bullet$ represents four arrows in the actual Auslander-Reiten quiver (each one connected to one of the $\OO(\x)$ and $\OO(\x+\bar{x}_{ab})$). 
	
	From Figure \ref{ARq of X}, we know that each indecomposable ACM bundle is either a line bundle or a $2$-extension bundle.  According to \cite{CRW}, the following objects are coincide: $2$-extension bundles, $2$-coextension bundles, rank-four indecomposable ACM bundles. 	
		
			\section{Rigid domains and roofs}\label{sec:Indecomposable ACM bundles of rank four}
		In this section, we introduce the notions of rigid domain and roof, together with their essential features, which play a key role in the construction of ACM tilting bundles. The major results of this paper are based on these properties.

		\subsection{Rigid domains} Throughout this subsection, let $E$ be an indecomposable ACM bundle of rank four.
		
			\begin{definition}  We define the \emph{rigid domain} of $E$, denoted by $\Dom(E)$ as the subset of $\ACM \X$ consisting of indecomposable ACM bundles $X$ such that $X \oplus E$ is rigid. 	Denote by $\Dom_{L}(E)$ the subset consisting of the line bundles in $\Dom(E)$. 
								
			Moreover, for a finite direct sum of rank-four indecomposable ACM bundle $F:=\bigoplus_{i\in I}E_i$, we denote by  $\Dom(F):=\bigcap_{i\in I}\Dom(E_i)$	and $\Dom_L(F):=\bigcap_{i\in I}\Dom_L(E_i)$.		
		\end{definition}
		
			\begin{lemma} \label{Hom-domains}
			Let  $X$ be an indecomposable ACM bundle. 
			\begin{itemize}
				\item[(a)]  If $\Hom(E,X)\neq 0$, then   we have $$\Hom(E,X(\x_4))\neq 0   \text{ and }  \Hom(E,X(-\w))\neq 0.$$
				\item[(b)]  If $\Hom(X,E)\neq 0$, then   we have $$\Hom(X(-\x_4),E)\neq 0  \text{ and } \Hom(X(\w),E)\neq 0.$$
			\end{itemize}
		\end{lemma}
		\begin{proof} We only prove (a), since  (b) can be shown by a dual argument.
			Concerning the rank of $X$, we divide the proof into the following two cases.
			
			\emph{Case $1$}:  $X$ is a line bundle, denoted by $L$.
			Applying $\Hom(E,-)$ to the injective morphism $L \to L(\x_4)$,  we obtain the first inequality. 
			
			For the second inequality, we let $L':=L(\x_1-\x_2)$. By Proposition \ref{L-action on 2-extension bundles}(b), we have $E\simeq E(\x_1-\x_2)$ and thus
			$ \Hom_{}(E,L')\simeq \Hom_{}(E,L) \neq 0.$
			Take a nonzero morphism $f:E \to L'$. Then the composition $if: E \to L' \to L'(\x_3+\x_4)$ of $f:E \to L'$ and the injective morphism $i:L' \to L'(\x_3+\x_4)$ is also  nonzero. Note that $L'(\x_3+\x_4)=L(-\w)$. Thus the second inequality follows.
			
			\emph{Case $2$}:  $X$ is an indecomposable ACM bundle with $\rk X=4$, denoted by $F$. Note that $\tau^{-1} F:= F(-\w)[-1] =F(\x_4)$, where $\tau$ denotes the Auslander–Reiten translation.
			Then there exists an almost split sequence $$0 \to F \to Y \to F(\x_4) \to 0$$ in $\ACM \X$. Since $\Hom_{}(E,F) \neq 0$  by our assumption, there is an indecomposable direct summand $Y'$ of $Y$ such that $\Hom_{}(E,Y') \neq 0$. Moreover, $\rk Y' \le 4=\rk F(\x_4)$ implies that the irreducible morphism $\phi: Y' \to F(\x_4)$ is injective. In fact, if $\phi$ is surjective, then its kernel has rank $0$ and thus $\phi$ is an isomorphism, a contradiction. Then  $\Hom_{}(E,F(\x_4))\neq 0$, which concludes the first inequality.  
			
			For the second inequality, let $\mathfrak{I}(F)$ be the injective hull of $F$ in the Frobenius category $\ACM \X$. By Proposition \ref{L-action on 2-extension bundles}(b), there exists an exact sequence $$0\to F \to \mathfrak{I}(F) \to F(\x_3) \to 0$$ in $\ACM \X$. Since $\Hom_{}(E,F) \neq 0$, there exists some line bundle $L''$ of $\mathfrak{I}(F)$ such that $\Hom_{}(E,L'') \neq 0$. 
			Then the composition $E \to L'' \to F(\x_3) $ of  a nonzero morphism $E \to L''$ and the injective morphism $L'' \to  F(\x_3)$ is also nonzero. 
			Moreover, by Proposition \ref{L-action on 2-extension bundles}, we have
			 $\Hom(E,F(-\w))\simeq \Hom(E,F(\x_3+\x_4))$, which is nonzero by the first inequality. Hence we have the assertion.			
		\end{proof}

		\begin{proposition}\label{Rigid-domains} Let $X$ be an indecomposable ACM bundle of rank four. Then $X\in \Dom(E)$ if and only if $X$ satisfies the following conditions.
			\begin{itemize}
				\item[(a)] $\Hom(X,E(\w))=0$ and $\Hom(E(-\w),X)=0$.
				\item[(b)] $\underline{\Hom}(E[-1],X)=0$ and $\underline{\Hom}(X,E[1])=0$. 
			\end{itemize}
		\end{proposition}
		
		\begin{proof} By the exceptional properties of $E$ and $X$,  $E \oplus X$ is rigid if and only if 	
			$$ \Ext_{}^{i}(E,X)=0\ \text{ and } \ \Ext_{}^{i}(X,E)=0  \ \text{ for } i=1,2. $$
			Using Auslander-Reiten-Serre duality, we have 
			\begin{align} \label{Rigid-domains equ. 0}
				\Hom_{}(X,E(\w))\simeq D\Ext_{}^2(E,X)\ \text{ and }\ \Hom_{}(E(-\w),X)\simeq D\Ext_{}^2(X,E).
			\end{align}	
					
			On the other hand, we claim the following equalities
			\begin{align} \label{Rigid-domains equ.}
				\underline{\Hom}(E[-1],X)\simeq \Ext^{1}(E,X)\ \text{ and } \ \underline{\Hom}(X,E[1])\simeq\Ext^{1}(X,E).
			\end{align}
			We only show the first isomorphism since the second one can be shown by a similar argument.
			If $X$ is a line bundle, then (\ref{Rigid-domains equ.}) is obvious by (\ref{ACM bundle eq.}). Otherwise we apply $\Hom(E,-)$ to the exact sequence $0\to X \to \mathfrak{I}(X) \xrightarrow{\phi} X[1] \to 0$, where $\mathfrak{I}(X)$ is the injective hull of $X$. Then we obtain an exact sequence $$ \Hom(E,\mathfrak{I}(X))\xrightarrow{\phi^{\ast}} \Hom(E,X[1]) \to \Ext^{1}(E,X) \to 0. $$
			Observe that for each $f\in\Hom(E,X[1])$, we have $f=0$ in $\underline{\Hom}(E,X[1])$ (that is, $f$ factors through a finite direct sum of line bundles) if and only if $f$ factors through the injective hull $\mathfrak{I}(X)$ of $X$. Thus
			we have $$\coker (\phi^{\ast})=\Hom(E,X[1])/ \Im \phi^{\ast} =\underline{\Hom}(E,X[1]).$$
			Thus  (\ref{Rigid-domains equ.}) holds. Combining (\ref{Rigid-domains equ. 0}) and (\ref{Rigid-domains equ.}), we obtain the assertion.
		\end{proof}
		
		Combining the previous proposition with  Lemma \ref{Hom-domains},  $\Dom(E)$ may be illustrated by 
		the following shaded regions of the Auslander-Reiten quiver ${\mathfrak A}(\ACM \X)$.
		\begin{figure}[H]
		\begin{tikzpicture}[x=0.75pt,y=0.85pt,yscale=-1.05,xscale=1.05] 
			\draw [draw opacity=0][fill={rgb, 255:red, 208; green, 2; blue, 27 }  ,fill opacity=0.2 ]   (276.43,132.01) -- (314.75,169.43) -- (237.87,168.72) -- cycle ;
			\draw [shift={(276.43,132.01)}, rotate = 316.4] [draw opacity=0][line width=0.75]      (0, 0) circle [x radius= 3.35, y radius= 3.35]   ;
			\draw [draw opacity=0][fill={rgb, 255:red, 208; green, 2; blue, 27 }  ,fill opacity=0.2 ]   (213.68,70.33) -- (340.02,70.75) -- (276.2,132.13) -- (213.68,70.75) ;
			\draw [shift={(213.68,70.75)}, rotate = 224.48] [draw opacity=0][line width=0.75]      (0, 0) circle [x radius= 3.35, y radius= 3.35]   ; 
			\draw [draw opacity=0][fill={rgb, 255:red, 74; green, 144; blue, 226 }  ,fill opacity=0.2 ]   (137.88,70.75) -- (213.68,70.75) -- (175.12,106.7) -- (237.87,168.72) -- (112.62,169.43) -- (175.28,106.37) -- cycle ;
			\draw [shift={(137.88,70.75)}, rotate = 223.61] [draw opacity=0][line width=0.75]      (0, 0) circle [x radius= 3.35, y radius= 3.35]   ;
			\draw [shift={(137.88,70.75)}, rotate = 0] [draw opacity=0][line width=0.75]      (0, 0) circle [x radius= 3.35, y radius= 3.35]   ; 
			\draw [draw opacity=0][fill={rgb, 255:red, 74; green, 144; blue, 226 }  ,fill opacity=0.2 ]   (340.25,70.44) -- (415.82,70.75) -- (376.84,107.05) -- (376.84,107.05) -- (440.01,168.72) -- (314.75,169.43) -- (376.84,107.05) -- cycle ;
			\draw [color={rgb, 255:red, 0; green, 0; blue, 0 }  ,draw opacity=1 ]   (62.04,70.75) -- (491.38,70.39) ; 
			\draw [color={rgb, 255:red, 0; green, 0; blue, 0 }  ,draw opacity=1 ]   (61.66,169.43) -- (491.01,169.07) ;
			\draw    (340.02,70.75) -- (238.95,169.43) ; 
			\draw    (213.68,70.75) -- (314.75,169.43) ;
			\draw  [dash pattern={on 0.84pt off 2.51pt}]  (213.68,70.75) -- (174.71,107.05) -- (237.87,168.72) ;
			\draw    (112.62,169.43) -- (174.71,107.05) -- (137.88,70.75) ;
			\draw  [dash pattern={on 0.84pt off 2.51pt}]  (314.75,169.43) -- (376.84,107.05) -- (340.02,70.75) ; 
			\draw    (415.82,70.75) -- (376.84,107.05) -- (440.01,168.72) ;
			
			\draw (281,128) node [anchor=north west][inner sep=0.75pt]  [font=\scriptsize]  {$E$};
			\draw (140,102) node [anchor=north west][inner sep=0.75pt]  [font=\scriptsize]  {$E[-1]$};
			\draw (382,102.5) node [anchor=north west][inner sep=0.75pt]  [font=\scriptsize]  {$E[1]$};
		\end{tikzpicture}
		\caption{$\Dom(E)$ in Auslander-Reiten quiver ${\mathfrak A}(\ACM \X)$}\label{Dom(E)}
		\end{figure}
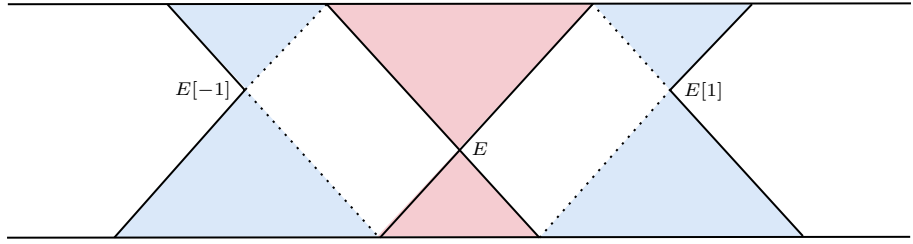
		
		For later use, we define a distinguished subset of $\Dom(E)$ by
		$$\Dom^{+}(E):=\{X\in \Dom(E)\mid \Hom_{}(X,E(-\x_4))=0 \text{ and } \Hom_{}(E(\x_4),X)=0\}.$$
		Its complement in $\Dom(E)$ is denoted by $\Dom^{-}(E):=\Dom(E)\setminus \Dom^{+}(E).$				
		Intuitively, the set $\Dom^{+}(E)$ (resp. $\Dom^{-}(E)$) corresponds to the red (resp. blue) region in Figure \ref{Dom(E)}.
		
		\begin{setting} \label{setting 1} We fix some conventions and notation. Any indecomposable  ACM bundle $E$ of rank four  is precisely a $2$-extension bundle, and it can be written uniquely as  $E=E\langle i\x_4 \rangle(k_i\x_4)$ with $0\le i\le p-2$ and $k_i\in\Z$. For simplicity, we just write the pair $\langle i,k_i\rangle$. Let $g\in G:=\{0,\x_1-\x_2,\x_1-\x_3,\x_2-\x_3\}$ and $\langle i,k_i\rangle_g:=\langle i,k_i\rangle(g)$. Note that $\langle i,k_i\rangle=\langle i,k_i\rangle_g$ holds by Proposition \ref{L-action on 2-extension bundles}.
			 			
			Enlarging the range of $i$ in $\langle i,k_i\rangle$ to the interval $[-1,p-1]$, 	we put
			\begin{align*}
				\langle -1,k_i \rangle:=\bigcup_{g\in G} \{\langle -1,k_i \rangle_g\} \ \text{ (resp. }  \langle p-1,k_i+1\rangle:=\bigcup_{g\in G} \{ \langle p-1,k_i+1  \rangle_g\} )
			\end{align*}
			where $\langle -1,k_i \rangle_g:=\OO(\w+k_i\x_4+g)$ and $\langle p-1,k_i+1 \rangle_g:=\OO(k_i\x_4+g)$.

				With the notation, we provide an explicit description of the following domains: 
			\begin{eqnarray}
				\Dom^{+}\langle i,k_i \rangle&=&\hspace{-1.5em}\bigcup_{{\substack{n_1n_2 \ge 0 \\ -i-1\le n_1+n_2\le p-1-i }}}\hspace{-1.5em}\{ \langle i+n_1+n_2,k_i-n_2 \rangle_g \mid g\in G\},\\ 	
				\Dom_{L}\langle i,k_i \rangle&=&\bigcup_{n_1\in I_1} \langle -1,n_1 \rangle \cup \bigcup_{n_2\in I_2} \langle p-1,n_2+1 \rangle \label{Dom 1},			
			\end{eqnarray}
			where the integer ranges 
			\begin{align*}I_1=[i+k_i-p+1,k_i+p] \ \text{ and }\ I_2=[k_i-p-1,i+k_i].\end{align*}	
				
			Moreover, we divide $\Dom^{+}\langle i,k_i \rangle$ into the following two useful parts 
				\begin{eqnarray} 
				\Dom^{+}_{\triangle}\langle i,k_i \rangle&:=&\hspace{-0.3em}\bigcup_{\substack{n_1\ge 0, n_2 \ge 0\\  n_1+n_2\le i+1}} \hspace{-0.3em}  \{ \langle i-n_1-n_2,k_i+n_2 \rangle_g  \mid g\in G\} \label{Dom 3},\\
				\Dom^{+}_{\bigtriangledown}\langle i,k_i \rangle&:=&\hspace{-0.8em}\bigcup_{\substack{n_1\ge 0, n_2 \ge 0\\n_1+n_2\le p-1-i}}\hspace{-0.8em} \{ \langle i+n_1+n_2,k_i-n_2 \rangle_g \mid g\in G \}. \label{Dom 2}  			
			\end{eqnarray}	
			
				Similarly, we divide $\Dom_{L}\langle i,k_i \rangle$ into the following two parts
			\begin{align} \label{Dom 4}
				\Dom_{\triangle,L}\langle i,k_i \rangle:=\bigcup_{n_1\in I_1} \langle -1,n_1 \rangle \ \ \text{and} \ \ \Dom_{\bigtriangledown,L}\langle i,k_i \rangle:=\bigcup_{n_2\in I_2} \langle p-1,n_2+1 \rangle.
			\end{align}

		\end{setting}	

		We give a simple example to explain our above setting.
		
		\begin{example} \label{example_1}
			Let $\X$ be a GL projective space of weight type $(2,2,2,4)$. Put $E:=\langle 1,0 \rangle$.  
			A piece of the Auslander-Reiten quiver $\mathfrak A{(\ACM \X)}$ is the following, where the objects in $\Dom (E)$ coincide with the red-shaded region $\Dom^{+}(E)$ and  the blue-shaded region $\Dom^{-}(E)$.  The objects in $\Dom^{+}_{\triangle}(E)$ (resp. $\Dom^{+}_{\bigtriangledown}(E)$) consist of $E$ together with  the objects in $\Dom^{+}(E)$  which posit on the below (resp. above) of $E$. 						
			The objects in $\Dom_L (E)$  correspond the line bundle contained in $\Dom (E)$, which consists of two parts: one is  the button line bundles $\Dom_{\triangle,L}(E)$ in $\Dom (E)$, and the other is the top line bundles $\Dom_{\bigtriangledown,L} (E)$ in $\Dom (E)$.
		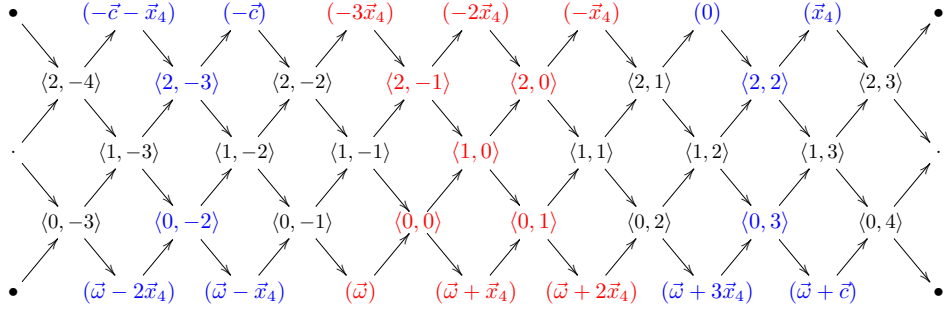
\begin{figure}[H]
				 \resizebox{\textwidth}{!}{
				\begin{xy} 0;<14pt,0pt>:<0pt,17pt>::
					(12,8) *+{\bullet} ="41",
					(12,4) *+{\cdot} ="22",
					(14,6) *+{\langle 2,-4 \rangle} ="32",
					(16,8) *+{\blue \scalebox{1.08}{$(-\c-\x_4)$}} ="42",
					(12,0) *+{\bullet} ="03",
					(14,2) *+{\langle 0,-3 \rangle} ="13",
					(16,4) *+{\langle 1,-3 \rangle} ="23",
					(18,6) *+{\blue \scalebox{1.08}{$\langle 2,-3 \rangle$}} ="33",
					(20,8) *+{\blue \scalebox{1.08}{$(-\c)$}} ="43",
					(16,0) *+{\blue \scalebox{1.08}{$(\w-2\x_4)$}} ="04",
					(18,2) *+{\blue  \scalebox{1.08}{$\langle 0,-2 \rangle$}} ="14",
					(20,4) *+{\langle 1,-2 \rangle} ="24",
					(22,6) *+{\langle 2,-2 \rangle} ="34",
					(24,8) *+{\red \scalebox{1.08}{$(-3\x_4)$}} ="44",
					(20,0) *+{\blue \scalebox{1.08}{$(\w-\x_4)$}} ="05",
					(22,2) *+{\langle 0,-1 \rangle} ="15",
					(24,4) *+{\langle 1,-1 \rangle} ="25",
					(26,6) *+{\red  \scalebox{1.08}{$\langle 2,-1 \rangle$}} ="35",
					(28,8) *+{\red \scalebox{1.08}{$(-2\x_4)$}} ="45",
					(24,0) *+{\red \scalebox{1.08}{$(\w)$}} ="06",
					(26,2) *{\red \scalebox{1.08}{$\langle 0,0 \rangle$}} ="16",
					(28,4) *+{\red \scalebox{1.08}{$\langle 1,0 \rangle$}} ="26",
					(30,6) *+{\red \scalebox{1.08}{$\langle 2,0 \rangle$}} ="36",
					(32,8) *+{\red \scalebox{1.08}{$(-\x_4)$}} ="46",
					(28,0) *+{\red \scalebox{1.08}{$(\w+\x_4)$}} ="07",
					(30,2) *+{\red \scalebox{1.08}{$\langle 0,1 \rangle$}} ="17",
					(32,4) *+{\langle 1,1 \rangle} ="27",
					(34,6) *+{\langle 2,1 \rangle} ="37",
					(36,8) *+{\blue \scalebox{1.08}{$(0)$}} ="47",
					(32,0) *+{\red \scalebox{1.08}{$(\w+2\x_4)$}} ="08",
					(34,2) *+{\langle 0,2 \rangle} ="18",
					(36,4) *+{\langle 1,2 \rangle} ="28",
					(38,6) *+{\blue \scalebox{1.08}{$\langle 2,2 \rangle$}} ="38",
					(40,8) *+{\blue \scalebox{1.08}{$(\x_4)$}} ="48",
					(36,0) *+{\blue \scalebox{1.08}{$(\w+3\x_4)$}} ="09",
					(38,2) *+{\blue \scalebox{1.08}{$\langle 0,3 \rangle$}} ="19",
					(40,4) *+{\langle 1,3 \rangle} ="29",
					(42,6) *+{\langle 2,3 \rangle} ="39",
					(44,8) *+{\bullet} ="49",
					(40,0) *+{\blue \scalebox{1.08}{$(\w+\c)$}} ="010",
					(42,2) *+{\langle 0,4 \rangle} ="110",
					(44,4) *+{\cdot} ="210",
					(44,0) *+{\bullet} ="011",
					"41", {\ar"32"},
					"22", {\ar"32"},
					"32", {\ar"23"},
					"32", {\ar"42"},
					"22", {\ar"13"},
					"42", {\ar"33"},
					"03", {\ar"13"},
					"13", {\ar"23"},
					"33", {\ar"43"},
					"13", {\ar"04"},
					"23", {\ar"14"},
					"23", {\ar"33"},
					"33", {\ar"24"},
					"43", {\ar"34"},
					"04", {\ar"14"},
					"14", {\ar"24"},
					"34", {\ar"44"},
					"14", {\ar"05"},
					"24", {\ar"15"},
					"24", {\ar"34"},
					"34", {\ar"25"},
					"44", {\ar"35"},
					"05", {\ar"15"},
					"15", {\ar"25"},
					"35", {\ar"45"},
					"35", {\ar"26"},
					"15", {\ar"06"},
					"25", {\ar"16"},
					"25", {\ar"35"},
					"45", {\ar"36"},
					"06", {\ar"16"},
					"16", {\ar"26"},
					"36", {\ar"46"},
					"36", {\ar"27"},
					"16", {\ar"07"},
					"26", {\ar"17"},
					"26", {\ar"36"},
					"46", {\ar"37"},
					"07", {\ar"17"},
					"17", {\ar"27"},
					"37", {\ar"47"},
					"37", {\ar"28"},
					"27", {\ar"37"},
					"17", {\ar"08"},
					"27", {\ar"18"},
					"47", {\ar"38"},
					"08", {\ar"18"},
					"18", {\ar"28"},
					"28", {\ar"38"},
					"38", {\ar"48"},
					"38", {\ar"29"},
					"18", {\ar"09"},
					"28", {\ar"19"},
					"48", {\ar"39"},
					"09", {\ar"19"},
					"19", {\ar"29"},
					"39", {\ar"49"},
					"39", {\ar"210"},
					"19", {\ar"010"},
					"29", {\ar"110"},
					"29", {\ar"39"},
					"010", {\ar"110"},
					"110", {\ar"210"},
					"110", {\ar"011"},
				\end{xy}
			} \caption{$\Dom^{+}(E)$ and $\Dom^{-}(E)$ in ${\mathfrak A}(\ACM \X)$}\label{Dom in AR}
		\end{figure}
		\end{example}
		
		As a direct consequence of Lemma \ref{Hom-domains}, we have the following description.
		\begin{lemma}  \label{Hom-domains 2}
			Let  $E:=\langle i,k_i \rangle$ and $X:=\langle j,k_j \rangle$ with $0\le i \le p-2$, $-1\le j \le p-1$  and $k_i,k_j\in \Z$. Then we have
			\begin{itemize}
				\item[(a)]  $\Hom(E,X)\neq0  \text{ if and only if } k_i\le k_j \text{ and } i+k_i\le j+k_j$.
				\item[(b)] $\Hom(X,E)\neq0  \text{ if and only if } k_j\le k_i \text{ and } j+k_j\le i+k_i$.
			\end{itemize}			
		\end{lemma}

			\begin{lemma}\label{Domain property} Let $X,Y$ be two indecomposable ACM bundles of rank four. The following statements are equivalent.
			\begin{itemize}
				\item[(a)] $X \in \Dom^{+}_{\triangle}(Y)$.
				\item[(b)] $\Dom^{+}_{\triangle}(X) \subset \Dom^{+}_{\triangle}(Y)$.
				\item[(c)] $Y\in \Dom^{+}_{\bigtriangledown}(X)$.
				\item[(d)] $\Dom^{+}_{\bigtriangledown}(Y) \subset \Dom^{+}_{\bigtriangledown}(X)$.
			\end{itemize}				
			In this case, we have $$\Dom_{\triangle,L}(Y)\subset \Dom_{\triangle,L}(X) \  \text{and}\  \Dom_{\bigtriangledown,L}(X)\subset \Dom_{\bigtriangledown,L}(Y).$$
		\end{lemma}
		\begin{proof} The first part follows directly from  (\ref{Dom 3}) and   (\ref{Dom 2}).
			For the second part, we only show the first inclusion, since the other one can be shown similarly. Note that there exist $0\le i,j \le p-2$ and $k_i,k_j\in \Z$, such that $X=\langle i,k_i \rangle$ and $Y=\langle j,k_j \rangle$. Since $X \in \Dom^{+}_{\triangle}(Y)$, we have $k_j\le k_i$ and $ i+k_i\le j+k_j$. By (\ref{Dom 4}), we have $\Dom_{\triangle,L}(Y)\subset \Dom_{\triangle,L}(X)$. Thus the assertion follows.
		\end{proof}

		\begin{corollary}\label{Domain corollary}
			Let $X,Y$ be two indecomposable ACM bundles of rank four. 
			\begin{itemize}
				\item[(a)] If $X \in \Dom^{+}_{\bigtriangledown}(E)$ and $Y \in \Dom^{+}_{\triangle}(E)$ hold, then $X \oplus Y$ is rigid.
				\item[(b)] We have $X\in \Dom^{\ast}(Y)$ if and only if $Y\in \Dom^{\ast}(X)$, where $\ast \in \{\emptyset,+,-\}$.
			\end{itemize}		
		\end{corollary}
		\begin{proof} (a) By Lemma \ref{Domain property}, $X \in \Dom^{+}_{\bigtriangledown}(E)$ implies $\Dom^{+}_{\triangle}(E) \subset \Dom^{+}_{\triangle}(X)$. Then we get $Y \in \Dom^{+}_{\triangle}(X)$ by our assumption. Thus the assertion follows. 
			
			(b) This is an immediate consequence of  Lemma \ref{Domain property}.
		\end{proof}

		\subsection{Left and right roofs} In this subsection, we present a method to generate an indecomposable ACM bundle of rank four from another such bundle together with certain line bundles. We introduce the following useful notation. 
		
		
		Throughout this section, fix an indecomposable ACM bundle $$E:=\langle i,k_i \rangle \text{\;\; with\;\;} 0\le i \le p-2,\; k_i\in \Z,$$ and         $$g\in G:=\{0,\x_1-\x_2,\x_1-\x_3,\x_2-\x_3\}.$$
		
		We introduce the following useful notation. For $\x \in \L$, the interval $[\x,\x-2\w]=[\x,\x+(p+2)\x_4]$ is called a \emph{roof} of $\x$. We consider the roofs contained in $\Dom_{L}(E)$. 
			 
		For  $1\le m \le i$ and $1\le n \le p-i-2$, setting		
				\begin{align*}
						\mathcal{L}^{}_m(E, g)&:=[l_m+g, l_m+(p+2)\x_4+g],\\
						\mathcal{R}^{}_{n}(E, g)&:=[r_n+g, r_n+(p+2)\x_4+g],						
					\end{align*}
					where $l_m:=(k_i+m-p-2)\x_4$ and $r_n:=\w+(i+k_i+n-p)\x_4$.
		We call $\mathcal{L}^{}_m(E,g)$ (resp. $\mathcal{R}^{}_{n}(E,g)$)  the $m$-th \emph{left roof} (resp. $n$-th \emph{right roof}) of the pair $(E, g)$. If $g=0$, we just write $\mathcal{L}^{}_m(E)$ (resp. $\mathcal{R}^{}_{n}(E)$).

			\begin{proposition}  \cite[Proposition 4.3]{CRW} \label{ext-pullback} 
			Assume $0\le \x \le \x+\y \le (p-2)\x_4$ and $\y=e\x_4$ with $e \ge 0$.
			Consider $y:=X_4^{e}$ as a morphism $L(\x)\to
			L(\x+\y)$. Then there exists a   commutative diagram 	
			\begin{equation*}\label{pullback and pushout diagram}
				\begin{tikzcd}[column sep=1.6em]
					\eta_{\x}:\quad\,\,	0\rar &	L(\w)\rar{\alpha}\arrow[d,equal]&E_L\langle \x \rangle \rar{\beta}\dar{}&L\langle \x \rangle  \rar{\gamma}\dar{y'}&L(\x)\arrow[d,"y"] \rar& 0\,\\
					\eta_{\x+\y}:\,\,	0\rar&	L(\w)\rar{\alpha}&E_{L}\langle \x+\y \rangle \rar{\beta}&L\langle \x+\y \rangle\rar{\gamma}&L(\x+\y) \rar& 0,
				\end{tikzcd}
			\end{equation*}
			with exact rows, where $y':={\rm diag}(y, y, y, 1)$.	
		\end{proposition}
			
		Recall that \emph{vector bundle duality} $$(-)^{\vee}: \vect\X \to \vect\X, \ X\mapsto \Hhom{}{X}{\OO},$$ 
		which sends line bundle $\OO(\x)$ to $\OO(-\x)$ for any $\x \in \L$.

			For simplicity, for a subset $I\subset \L$, we write $$\OO(I):=\{\OO(\x) \mid \x\in I\}.$$
			
		\begin{lemma}\label{ext-bundle thick}  For $0\leq i\leq p-2$,  the following assertions hold.
			
			\begin{itemize} 
				\item[(a)] If $i\neq 0$, then we have 
				$$\langle i-1,k_i\rangle \in \thick\{E, \OO(\mathcal{L}_{i}(E, g))\} \text{\;\;and\;\;} \langle i-1,k_i+1\rangle \in \thick\{E, \OO(\mathcal{L}_{1}(E, g))\}.$$
				\item[(b)] If $i\neq p-2$, then we have  $$\langle i+1,k_i\rangle \in \thick\{E,\OO(\mathcal{R}_{1}(E, g))\} 
				\text{\;\;and\;\;}\langle i+1,k_i-1\rangle \in \thick\{E, \OO(\mathcal{R}_{p-i-2}(E, g))\}.$$
			\end{itemize}
		\end{lemma}
		\begin{proof} We only show (a) since (b) can be shown dually. Without loss of generality, we can assume $g=0$ and $k_i=0$. In this case, $E=\langle i,0 \rangle$.

			Applying Proposition \ref{ext-pullback} to the data $L=\OO$, $\x=(i-1)\x_4$ and $\y=\x_4$, we obtain the following commutative diagram.
			\begin{equation} \label{ext-bundle diag 1}
				\begin{tikzcd}[column sep=1.6em]
						0\rar &	\OO(\w)\rar{}\arrow[d,equal]&\langle i-1,0 \rangle \rar{}\dar{}
						&\OO\langle (i-1)\x_4 \rangle  \rar{}\dar{}&\OO((i-1)\x_4)\arrow[d,"X_4"] \rar& 0\,\\
						0\rar&	\OO(\w)\rar{}& \langle i,0 \rangle \rar{}
						&\OO \langle i\x_4 \rangle \rar{}&\OO(i\x_4) \rar& 0
				\end{tikzcd}
			\end{equation}		
			It is straightforward to check that each object in the rightmost square belongs to $\add \OO(\mathcal{L}_{i}(E))$, which implies that $\OO(\w) \in \thick\{E, \OO(\mathcal{L}_{i}(E))\}$.  Therefore we have $\langle i-1,0\rangle \in \thick\{E, \OO(\mathcal{L}_{i}(E))\}$.

			By applying vector bundle duality and the degree shift by $\w+i\x_4$ to the diagram (\ref{ext-bundle diag 1}), we obtain a commutative diagram
			\begin{equation*}
				\begin{tikzcd}[column sep=1.6em]
					0\rar &	\OO(\w)\rar{}\ar[d,"{X_4}"]&\OO[i\x_4] \rar{}\dar{}
					&  F\langle i\x_4 \rangle \rar{}\dar{}&\OO(i\x_4)\arrow[d,equal] \rar& 0\,\\
					0\rar&	\OO(\w+\x_4) \rar{}& \OO[(i-1)\x_4](\x_4) \rar{}
					&F \langle (i-1)\x_4 \rangle(\x_4) \rar{}&\OO(i\x_4)\ar[r]& 0.
				\end{tikzcd}
			\end{equation*}
			 Note that we have $F\langle \x \rangle (\w+\x_4) \simeq E\langle \x \rangle$ by Proposition \ref{L-action on 2-extension bundles}. It is easy to check that each object in the leftmost square by applying degree shift $\w+\x_4$  belongs to $\add \OO(\mathcal{L}_{1}(E))$, which implies that $\OO(\w+(i+1)\x_4) \in \thick\{\langle i,0\rangle, \OO(\mathcal{L}_{1}(E))\}$.  Hence we have $\langle i-1,1\rangle \in \thick\{\langle i,0\rangle, \OO(\mathcal{L}_{1}(E)) \}$, and the assertion follows.
		\end{proof}
		
		\begin{corollary} \label{ext-generated}   Let $0\leq i\leq p-2$. Then  the following assertions hold.
			\begin{itemize} 
				\item[(a)] Assume $i\neq 0$. For any  $n_1,n_2 \ge 0$ with $n_1+n_2\le i$, we have $$\langle i-n_1-n_2,k_i+n_2\rangle \in \thick\{E, \OO(\bigcup_{\ell=1}^{n_2} \mathcal{L}_{\ell}(E, g)),  \OO(\bigcup_{\ell=1}^{n_1} \mathcal{L}_{i+1-\ell}(E, g))\}.$$
				\item[(b)]  Assume $i\neq p-2$.
				For any  $n_1,n_2 \ge 0$ with $n_1+n_2\le p-i-2$, we have $$\langle i+n_1+n_2,k_i-n_2 \rangle \in \thick\{E, \OO( \bigcup_{\ell=1}^{n_1} \mathcal{R}_{\ell}(E, g)),  \OO(\bigcup_{\ell=1}^{n_2} \mathcal{R}_{p-i-1-\ell}(E, g))\}.$$
			\end{itemize}
		\end{corollary}
		\begin{proof} 
			We only show (a), as (b) can be shown by a dual argument.
			
			(a) We prove the assertion by induction on $n:=n_1+n_2$. 			
			If $n=0$, this is nothing to show. For $n \ge 1$, without loss of generality we can assume  $n_1 \ge 1$. Let $F:=\langle i-n+1,k_i+n_2\rangle$. By Lemma \ref{ext-bundle thick}(a), we have 
			\begin{align} \label{ext-generated 1}
				\langle i-n,k_i+n_2\rangle \in \thick \{ F, \mathcal{L}_{i-n_1+1}(E,g) \}.
			\end{align}			
		 Indeed, $\mathcal{L}_{i-n_1+1}(E,g)=\mathcal{L}_{i-n+1}(F, g)$. By the induction hypothesis, we have 
		 \begin{align} \label{ext-generated 2}
		 	F \in \thick\{E, \OO(\bigcup_{\ell=1}^{n_2} \mathcal{L}_{\ell}(E, g)), \OO(\bigcup_{\ell=1}^{n_1-1} \mathcal{L}_{i+1-\ell}(E, g))\}.
		 \end{align}
			Combining (\ref{ext-generated 1}) and (\ref{ext-generated 2}), we have the assertion. 
		\end{proof}
		
		The union of all left (resp. right) roofs of the pair $(E,g)$ is denoted by
		\begin{align} \label{the union of roofs}
			\mathcal{L}_{g}(E):=\bigcup_{m=1}^{i} \mathcal{L}_m(E,g) \ \ \text{(resp.} \ \ \mathcal{R}_{g}(E):=\bigcup_{n=1}^{p-i-2} \mathcal{R}_{n}(E,g)).
		\end{align}
		If $g=0$, we just write $\mathcal{L}^{}(E)$ (resp. $\mathcal{R}^{}(E)$) for simplicity. 
		
		\begin{proposition} \label{reback lemma}  For $0\leq i\leq p-2$,  the following assertions hold.
				\begin{itemize} 
				\item[(a)] If $i\neq 0$, then	$\Dom^{+}_{\triangle}(E)\subset \thick\{E, \OO(\mathcal{L}_{g}(E))\}$.
				\item[(b)] If $i\neq p-2$, then	$\Dom^{+}_{\bigtriangledown}(E)\subset \thick\{E, \OO(\mathcal{R}_{g}(E))\}$.		 
			\end{itemize}
		\end{proposition}
		\begin{proof} We only show (a), as (b) can be shown dually. Without loss of generality, we can assume  $k_i=0$ and $g=0$. In this case,  we write $E=\langle i,0 \rangle$.   
			
			(a) Thanks to Corollary \ref{ext-generated}(a), it suffices to show that 
			\begin{align*} \label{claim_reback lemma}
				\langle -1,j \rangle \subset \thick\{E, \OO(\mathcal{L}^{}(E))\} \ \text{ for } 0\le j \le i+1.
			\end{align*}			
			By definition of $\mathcal{L}^{}(E)$, we have $\mathcal{L}^{}(E)=[-(p+1)\x_4, i\x_4]$.  Since the interval
			$$[-(p+1)\x_4 +\x_k, i\x_4 -\x_k] \subset \mathcal{L}^{}(E) \ \text{ for } 1\le k \le 3,$$
		  we have $\w+j\x_4+f\in \mathcal{L}^{}(E)$  for any $0\le j \le i+1$ and $0 \neq f\in G$. It remains to show that  $\OO(\w+j\x_4) \in \thick\{E, \OO(\mathcal{L}^{}(E))\}$ holds for $0\le j \le i+1$. 
		  		  
		  Applying degree shift by $j\x_4$ to the defining sequence (\ref{2-extension bundle exact})  for the $2$-Auslander bundle $\langle 0,0 \rangle$, we obtain an exact sequence $\eta(j\x_4)$ for $1\le j \le i$ as follows: 
		  		  $$ \eta(j\x_4): \  0\to \OO(\w+j\x_4) \to \langle 0 ,j \rangle \to \bigoplus_{1\le k \le 4} \OO(j\x_4-\x_k) \to \OO(j\x_4) \to 0.$$ 
		  	  Since all terms of $\eta(j\x_4)$ except $\OO(\w+j\x_4)$ belong to $\thick\{E, \OO(\mathcal{L}^{}(E))\}$,  we have $\OO(\w+j\x_4)\in \thick\{E, \OO(\mathcal{L}^{}(E))\}$ for any $0\le j \le i$. 
		  	  
		  	  On the other hand,  applying degree shift by $\w+\x_4$ to the defining sequence (\ref{2-coextension bundle exact})  for the $2$-coAuslander bundle $F\langle 0 \rangle$, we obtain an exact sequence 
		  	  $$
		  	  \mu: \	0\to \OO(-(p+1)\x_4) \to  \bigoplus_{1\le k \le 4} \OO(-(p+1)\x_4+\x_k) \to \langle 0 ,0 \rangle \to \OO(\w+\x_4) \to 0. $$
		  	  Since all terms of $\mu(i\x_4)$ except $\OO(\w+(i+1)\x_4)$ belong to $\thick\{E, \OO(\mathcal{L}^{}(E))\}$,  we have $\OO(\w+(i+1)\x_4)\in \thick\{E, \OO(\mathcal{L}^{}(E))\}$. Thus the assertion follows.	
		\end{proof}

		\begin{example} \label{example 2} 
			We continue to discuss Example \ref{example_1}. Let $E:=\langle 1,0 \rangle$. By definition the left roof $\mathcal{L}^{}(E)=[-\c-\x_4,\x_4]$ and the right roof $\mathcal{R}^{}(E)=[\w-2\x_4,\w+\c]$. They can be illustrated by the following Auslander-Reiten quiver $\mathfrak A{(\ACM \X)}$, where  the  line bundles over $\mathcal{L}^{}(E)$ (resp. $\mathcal{R}^{}(E)$) are colored in blue (resp. red) as follows:
			\begin{figure}[H]		   
		 \resizebox{\textwidth}{!}{
			\begin{xy} 0;<14pt,0pt>:<0pt,17pt>::
				(12,8) *+{\bullet} ="41",
				(12,4) *+{\cdot} ="22",
				(14,6) *+{\langle 2,-4 \rangle} ="32",
				(16,8) *+{ \scalebox{1}{$(-\c-\x_4)$}} ="42",
				(12,0) *+{\bullet} ="03",
				(14,2) *+{\langle 0,-3 \rangle} ="13",
				(16,4) *+{\langle 1,-3 \rangle} ="23",
				(18,6) *+{ \scalebox{1}{$\langle 2,-3 \rangle$}} ="33",
				(20,8) *+{ \scalebox{1}{$(-\c)$}} ="43",
				(16,0) *+{\scalebox{1}{$(\w-2\x_4)$}} ="04",
				(18,2) *+{  \scalebox{1}{$\langle 0,-2 \rangle$}} ="14",
				(20,4) *+{\langle 1,-2 \rangle} ="24",
				(22,6) *+{\langle 2,-2 \rangle} ="34",
				(24,8) *+{\scalebox{1}{$(-3\x_4)$}} ="44",
				(20,0) *+{\scalebox{1}{$(\w-\x_4)$}} ="05",
				(22,2) *+{\langle 0,-1 \rangle} ="15",
				(24,4) *+{\langle 1,-1 \rangle} ="25",
				(26,6) *+{ \scalebox{1}{$\langle 2,-1 \rangle$}} ="35",
				(28,8) *+{ \scalebox{1}{$(-2\x_4)$}} ="45",
				(28,8) *+{ } ="400",
				(24,0) *+{ \scalebox{1}{$(\w)$}} ="06",
				(26,2) *{ \scalebox{1}{$\langle 0,0 \rangle$}} ="16",
				(28,4) *+ [Fo]{{\begin{smallmatrix} \scalebox{1}{$\langle 1,0 \rangle$}\end{smallmatrix}}}="26",
				(30,6) *+{ \scalebox{1}{$\langle 2,0 \rangle$}} ="36",
				(32,8) *+{ \scalebox{1}{$(-\x_4)$}} ="46",
				(28,0) *+{ \scalebox{1}{$(\w+\x_4)$}} ="07",
				(28,0) *+{ } ="300",
				(30,2) *+{ \scalebox{1}{$\langle 0,1 \rangle$}} ="17",
				(32,4) *+{\langle 1,1 \rangle} ="27",
				(34,6) *+{\langle 2,1 \rangle} ="37",
				(36,8) *+{ \scalebox{1}{$(0)$}} ="47",
				(32,0) *+{ \scalebox{1}{$(\w+2\x_4)$}} ="08",
				(34,2) *+{\langle 0,2 \rangle} ="18",
				(36,4) *+{\langle 1,2 \rangle} ="28",
				(38,6) *+{ \scalebox{1}{$\langle 2,2 \rangle$}} ="38",
				(40,8) *+{ \scalebox{1}{$(\x_4)$}} ="48",
				(36,0) *+{ \scalebox{1}{$(\w+3\x_4)$}} ="09",
				(38,2) *+{ \scalebox{1}{$\langle 0,3 \rangle$}} ="19",
				(40,4) *+{\langle 1,3 \rangle} ="29",
				(42,6) *+{\langle 2,3 \rangle} ="39",
				(44,8) *+{\bullet} ="49",
				(40,0) *+{ \scalebox{1}{$(\w+\c)$}} ="010",
				(42,2) *+{\langle 0,4 \rangle} ="110",
				(44,4) *+{\cdot} ="210",
				(44,0) *+{\bullet} ="011",
				"41", {\ar"32"},
				"22", {\ar"32"},
				"32", {\ar"23"},
				"32", {\ar"42"},
				"22", {\ar"13"},
				"42", {\ar"33"},
				"03", {\ar"13"},
				"13", {\ar"23"},
				"33", {\ar"43"},
				"13", {\ar"04"},
				"23", {\ar"14"},
				"23", {\ar"33"},
				"33", {\ar"24"},
				"43", {\ar"34"},
				"04", {\ar"14"},
				"14", {\ar"24"},
				"34", {\ar"44"},
				"14", {\ar"05"},
				"24", {\ar"15"},
				"24", {\ar"34"},
				"34", {\ar"25"},
				"44", {\ar"35"},
				"05", {\ar"15"},
				"15", {\ar"25"},
				"35", {\ar"45"},
				"35", {\ar"26"},
				"15", {\ar"06"},
				"25", {\ar"16"},
				"25", {\ar"35"},
				"45", {\ar"36"},
				"06", {\ar"16"},
				"16", {\ar"26"},
				"36", {\ar"46"},
				"36", {\ar"27"},
				"16", {\ar"07"},
				"26", {\ar"17"},
				"26", {\ar"36"},
				"46", {\ar"37"},
				"07", {\ar"17"},
				"17", {\ar"27"},
				"37", {\ar"47"},
				"37", {\ar"28"},
				"27", {\ar"37"},
				"17", {\ar"08"},
				"27", {\ar"18"},
				"47", {\ar"38"},
				"08", {\ar"18"},
				"18", {\ar"28"},
				"28", {\ar"38"},
				"38", {\ar"48"},
				"38", {\ar"29"},
				"18", {\ar"09"},
				"28", {\ar"19"},
				"48", {\ar"39"},
				"09", {\ar"19"},
				"19", {\ar"29"},
				"39", {\ar"49"},
				"39", {\ar"210"},
				"19", {\ar"010"},
				"29", {\ar"110"},
				"29", {\ar"39"},
				"010", {\ar"110"},
				"110", {\ar"210"},
				"110", {\ar"011"},
				%
				%
				%
				%
				%
					(16,8.8) *+{\blue \scalebox{1}{$-\c-\x_4$}},(15.4,9.6) *+{\scalebox{1}{$-\x_2-\x_3-\x_4$}},(15.4,10.4) *+{\scalebox{1}{$-\x_1-\x_3-\x_4$}},(15.4,11.2) *+{ \scalebox{1}{$-\x_1-\x_2-\x_4$}},
				(20,8.8) *+{\blue\scalebox{1}{$-\c$}},(20,9.6) *+{\scalebox{1}{$-\x_2-\x_3$}},(20,10.4) *+{\scalebox{1}{$-\x_1-\x_3$}},(20,11.2) *+{\scalebox{1}{$-\x_1-\x_2$}},
				(24,8.8) *+{\blue\scalebox{1}{$-3\x_4$}},(24,9.6) *+{\red\scalebox{1}{$\bar{x}_{23}-3\x_4$}},(24,10.4) *+{ \red\scalebox{1}{$\bar{x}_{13}-3\x_4$}},(24,11.2) *+{\color{red}\scalebox{1}{$\bar{x}_{12}-3\x_4$}},
				(28,8.8) *+{\blue\scalebox{1}{$-2\x_4$}},(28,9.6) *+{\red\scalebox{1}{$\bar{x}_{23}-2\x_4$}},(28,10.4) *+{\red\scalebox{1}{$\bar{x}_{13}-2\x_4$}},(28,11.2) *+{\color{red}\scalebox{1}{$\bar{x}_{12}-2\x_4$}},
				(32,8.8) *+{\blue\scalebox{1}{$-\x_4$}},(32,9.6) *+{\red\scalebox{1}{$\bar{x}_{23}-\x_4$}},(32,10.4) *+{\red\scalebox{1}{$\bar{x}_{13}-\x_4$}},(32,11.2) *+{\color{red}\scalebox{1}{$\bar{x}_{12}-\x_4$}},
				(36,8.8) *+{\blue\scalebox{1}{$0$}},(36,9.6) *+{\scalebox{1}{$\bar{x}_{23}$}},(36,10.4) *+{\scalebox{1}{$\bar{x}_{13}$}},(36,11.2) *+{\scalebox{1}{$\bar{x}_{12}$}},
				(40,8.8) *+{\blue\scalebox{1}{$\x_4$}},(40,9.6) *+{\scalebox{1}{$\bar{x}_{23}+\x_4$}},(40,10.4) *+{\scalebox{1}{$\bar{x}_{13}+\x_4$}},(40,11.2) *+{\scalebox{1}{$\bar{x}_{12}+\x_4$}},			
				(16,-0.8) *+{\scalebox{1}{$-\x_3-3\x_4$}},(16,-1.6) *+{\scalebox{1}{$-\x_2-3\x_4$}},(16,-2.4) *+{\scalebox{1}{$-\x_1-3\x_4$}},(16,-3.2) *+{\red \scalebox{1}{$\w-2\x_4$}},
				(20,-0.8) *+{\scalebox{1}{$-\x_3-2\x_4$}},(20,-1.6) *+{\scalebox{1}{$-\x_2-2\x_4$}},(20,-2.4) *+{\scalebox{1}{$-\x_1-2\x_4$}},(20,-3.2) *+{\red \scalebox{1}{$\w-\x_4$}},
				(24,-0.8) *+{\blue\scalebox{1}{$-\x_3-\x_4$}},(24,-1.6) *+{\blue\scalebox{1}{$-\x_2-\x_4$}},(24,-2.4) *+{\blue\scalebox{1}{$-\x_1-\x_4$}},
				(24,-3.2) *+{\color{red} \scalebox{1}{$\w$}},
				(28,-0.8) *+{\blue\scalebox{1}{$-\x_3$}},(28,-1.6) *+{\blue\scalebox{1}{$-\x_2$}},(28,-2.4) *+{\blue\scalebox{1}{$-\x_1$}},(28,-3.2) *+{\color{red} \scalebox{1}{$\w+\x_4$}},
				(32,-0.8) *+{\blue \scalebox{1}{$-\x_3+\x_4$}},(32,-1.6)*+{\blue\scalebox{1}{$-\x_2+\x_4$}},(32,-2.4) *+{\blue\scalebox{1}{$-\x_1+\x_4$}},(32,-3.2) *+{\color{red} \scalebox{1}{$\w+2\x_4$}},
				(36,-0.8) *+{\scalebox{1}{$-\x_3+2\x_4$}},(36,-1.6) *+{\scalebox{1}{$-\x_2+2\x_4$}},(36,-2.4) *+{\scalebox{1}{$-\x_1+2\x_4$}},(36,-3.2) *+{\red \scalebox{1}{$\w+3\x_4$}},
				(40,-0.8) *+{\scalebox{1}{$-\x_3+3\x_4$}},(40,-1.6) *+{\scalebox{1}{$-\x_2+3\x_4$}},(40,-2.4) *+{\scalebox{1}{$-\x_1+3\x_4$}},(40,-3.2) *+{\red \scalebox{1}{$\w+\c$}},
			\end{xy}			
		} \caption{The roofs $\mathcal{L}^{}(E)$ and $\mathcal{R}^{}(E)$ in ${\mathfrak A}(\ACM \X)$}\label{the roofs in Dom}
		\end{figure}
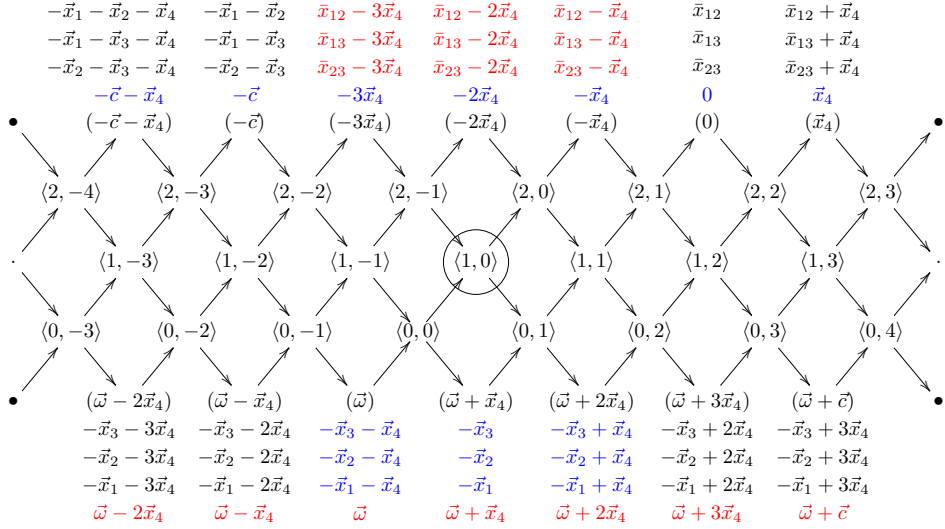

		By Proposition \ref{reback lemma}, we have $\Dom^{+}_{\triangle}(E)\subset \thick\{E, \OO(\mathcal{L}(E))\}$ and $\Dom^{+}_{\bigtriangledown}(E)\subset \thick\{E, \OO(\mathcal{R}(E))\}$. 
		Let $g:=\x_1-\x_2$ and $h=0$. Moreover, the object 
		\begin{equation} \label{Texample 2}
			T:= E \oplus \big(\bigoplus_{{\x \in \mathcal{L}_{g}(E)\cup R_{h}(E)}} \OO(\x)\big)
		\end{equation}
		 is a tilting object on $\X$, see Theorem \ref{main theorem}.		
		\end{example}

			\section{ACM tilting bundles on $\X$}\label{sec:ACM tilting bundles on}
			In this section, we investigate  ACM tilting bundles on $\X$. We first prove that a tilting bundle consisting of line bundles is the $2$-canonical tilting bundle up to degree shift.  We also provide a program to construct a family of ACM tilting bundles, and 
			show that such ACM tilting bundles are characterized by a certain condition among all ACM tilting bundles.
			
			\subsection{Tilting bundles consisting of line bundles} 
			Recall from \cite{HIMO} that  \[T^{\ca}:=\bigoplus\limits_{0\leq \x\leq 2\c}\OO(\x)\] is a tilting bundle on $\X$, called \emph{$2$-canonical tilting bundle},  whose endomorphism algebra is a $2$-canonical algebra of the same type as $\X$.
			
			\begin{theorem}\label{ACM tilting bundles consisting of line bundles}
				Let $T$ be a tilting object on $\X$ consisting of line bundles.
				Then T is isomorphic to $T^{\ca}$ up to degree shift, that is,
				\[T=\bigoplus\limits_{0\leq\x\leq2\c}L(\x), ~\text{\ for\ some\ line bundle\ } L.\]
			\end{theorem}
			
			To prove Theorem \ref{ACM tilting bundles consisting of line bundles}, we need the following observations.
			
			\begin{lemma}\label{rigid domain among line bundles}
				For any line bundle $L$ and $\x, \y \in \L$, 
				\[ L(\x) \oplus L(\y) \text{\ is\ rigid} \text{\ if and only if} -2\c\leq \y- \x\leq 2\c. \]
			\end{lemma}
			
			\begin{proof}
				
				By Theorem \ref{basic properties}(c),  it suffices to show that $\Ext_{}^2(L(\x),L(\y))=0$ is equivalent to $-2\c\leq \y-\x$, and $\Ext_{}^2(L(\y),L(\x))=0$ is equivalent to $\y- \x\leq 2\c$. We only prove the former one, as the latter  can be proved similarly.
			
				Using Auslander-Reiten-Serre duality, we have $\Ext_{}^2(L(\x),L(\y))=D(R_{\x-\y+\w})$. By (\ref{two possibilities of x}), this is zero if and only if $\x-\y+\w \leq \vec{\omega}+2\c$, that is, $-2\c\leq \y- \x$.
			\end{proof}
			
			\begin{corollary} \label{rigid domain among line bundles cor} 
				For a line bundle $L$ and an integer $\ell$, the following assertions hold.
				\begin{itemize}
					\item[(a)]  $L \oplus L(\ell\w)$ is rigid if and only if $\ell\in \{-2, 0, 2 \}$.
					\item[(b)]  $L \oplus L(\pm\w+ \ell\x_4)$ is rigid if and only if $1 \le\pm\,\ell \le p+1$. 
					\item[(c)]  For any $i,j$ with $1\le i < j \le 3$, $L \oplus L(\x_i-\x_j +\ell\x_4)$ is rigid if and
					only if $-p \le\ell \le p$. 
					
				\end{itemize}		
			\end{corollary}
			
			\begin{proof} (a) By Lemma \ref{rigid domain among line bundles}, $L \oplus L(\ell\w)$ is rigid if and only if $-2\c\leq \ell\w\leq 2\c$. We divide the proof into the cases of $\ell$ even or odd. 
				
				Suppose $\ell=2k$ for some integer $k$. We have $-2\c\leq 2k(\c-\sum_{i=1}^{4}\x_i)\leq 2\c$, which yields $-2p\x_4\leq k(p+2)\x_4\leq 2p\x_4$, which is equivalent to  
				$-\frac{2p}{p+2} \le k \le \frac{2p}{p+2}$, that is, $-1 \le k \le 1$. Equivalently, $\ell\in \{-2, 0, 2 \}$.  
				
				Suppose $\ell=2k+1$ for some integer $k$. We have $-2\c\leq (2k+1)(\c-\sum_{i=1}^{4}\x_i)\leq 2\c$, which by a simple calculation implies that $-2\c\leq (1-k)\c-\sum_{i=1}^{3}\x_i-(2k+1)\x_4\leq 2\c$. This is equivalent to  $k\c\leq -(2k+1)\x_4\leq (k+1)\c$, yielding that $-\frac{p+1}{p+2} \le k \le -\frac{1}{p+2}$,  a contradiction.					Combining the two cases, we have the statement (a).
				
				(b)(c) Immediate from Lemma \ref{rigid domain among line bundles}.
			\end{proof}
			
			Now we are ready to prove Theorem \ref{ACM tilting bundles consisting of line bundles}.
			\begin{proof}[Proof of Theorem~\rm\ref{ACM tilting bundles consisting of line bundles}.] By Corollary \ref{rigid domain among line bundles cor}(a), we know that $\add T$ meets each $w$-orbit of line bundles at most twice. Note that $|T|=5p+7$ and $|\L/\Z\w|=4p+8$. Then there are at least $p-1$ $w$-orbits of line bundles contain exactly two elements in $\add T$. Thus we may assume that $T$ have the following direct summands
				$$L_1 \oplus L_1(-2\w), L_2 \oplus L_2(-2\w), \dots ,L_{p-1} \oplus L_{p-1}(-2\w).$$
				Then by Lemma \ref{rigid domain among line bundles}, for any $i,j \in \{1,
				\dots, p-1\}$, $L_i \oplus L_i(2\w)\oplus L_j \oplus L_j(2\w)$	is rigid if and only if $L_j=L_i(\x)$ with $-(p-2)\x_4 \le \x \le (p-2)\x_4$. Rearranging the indices of $L_1,\dots,L_{p-1}$, we may suppose that $L_{i+1}=L_{i}(\x_4)$ for any $1\le i \le p-2$, and conclude that $L_{p-1}(-2\w)=L_1(2\c)$. Since $L_1$ and $ L_1(2\c)$ belong to $\add T$, then  by Lemma \ref{rigid domain among line bundles}, the possible indecomposable direct summands of $T$ are cantained in 
				$\{L_1(\x) \mid 0\le\x\le 2\c\}$.  This implies the assertion by cardinality reason. 	
			\end{proof}

			\subsection{A class of ACM tilting bundles on $\X$} In this subsection, we study a class of ACM tilting bundles with rank-four indecomposable direct summands, and then give the trichotomy of the form for each such ACM tilting bundles.
			
			  Throughout this subsection, we let $E_i=\langle i, k_i \rangle$ and $E_j=\langle j, k_j \rangle$ be any  two rank-four indecomposable ACM bundles with $0\le i\le j\le p-2$ and $k_i,k_j \in\Z$ such that $E_i \in \Dom^{+}(E_j)$. By Lemma \ref{Domain property}, we have   
			\begin{align}\label{cap E_i,j}
				\bigcap_{n=i,j}\Dom_{L}(E_n)=\bigcup_{n_1\in I_1} \langle -1,n_1 \rangle \cup \bigcup_{n_2\in I_2} \langle p-1,n_2+1 \rangle,
			\end{align}
			where $I_1=[j+k_j-p+1,k_j+p]$ and $ I_2=[k_i-p-1,i+k_i]$.

			\begin{lemma} \label{two or three elements}
				Under the above setting, each $\w$-orbit of line bundles contains either two or three elements in $\bigcap_{n=i,j}\Dom_{L}(E_n)$.				
			\end{lemma}
			\begin{proof} Since $E_i \in \Dom_{\triangle}^{+}(E_j)$, we have $k_j\le k_i$ and $i+k_i\le j+k_j$.
				By a simple calculation, we have
				\begin{align*}
					\{k_i+1, i+k_i\} \subset I_1 \ \text{ and }\ \{j+k_j-p+1, k_j-2 \} \subset I_2.
				\end{align*}
				This implies that each $\w$-orbit of line bundles contains at least two elements in $\bigcap_{n=i,j}\Dom_{L}(E_n)$.	
				Since $k_i-p-1+2(p+2)>k_j+p$ and $i+k_i-2(p+2)<j+k_j-p+1$, we have that each $\w$-orbit of line bundles contains at most three elements in $\bigcap_{n=i,j}\Dom_{L}(E_n)$.	Thus the assertion follows.									
			\end{proof}
			
			Let $L$ be a line bundle and $G:=\{0,\x_1-\x_2,\x_1-\x_3,\x_2-\x_3\}$. For an integer $k\ge 0$, we call the set
			\begin{align} \label{segment}
		 \bigcup_{0\le \ell\le k}\{L(g+\ell\w)\mid g\in G \}
		\end{align}
		 \emph{$(\w,k)$-segment} associated with $L$.

			\begin{lemma} \label{piece} 	
				Let $V$ be a $(\w,1)$-segment (resp. $(\w,2)$-segment). Then there are at most $4$ (resp. $5$)  line bundles of $V$ such that their direct sum is rigid.			
			\end{lemma}
			\begin{proof} In the case when $V$ is a $(\w,1)$-segment, the assertion follows immediately from Corollary \ref{rigid domain among line bundles cor}(a). Now we consider the case when $V$ is a  $(\w,2)$-segment. By definition, there exists a line bundle $L$ such that 
				$$V=\bigcup_{0\le i\le 2}V_i,\   \ \text{where} \ V_i:= \{L(g+i\w)\mid g\in G\}.$$
			  Let $M$ be a subset of $V$ such that the direct sum of all elements in $M$ is rigid. We claim that $|M|\le 5$. For $i=0,1,2$, we let $a_i=|M\cap V_i|$. 
				By Corollary \ref{rigid domain among line bundles cor}(a), we have $a_1+a_2\le 4$ and $a_2+a_3\le 4$. By Corollary \ref{rigid domain among line bundles cor}(c), we know that $a_1\ge 2$ implies $a_3=0$, and  $a_3\ge 2$ implies $a_1=0$. Thus we have 
				$$|M|=a_1+a_2+a_3\le 1+3+1=5.$$	Hence we have the assertion.	
			\end{proof}
			
			\begin{remark}\label{double piece} Combining the argument of the above proof and Lemma \ref{rigid domain among line bundles},  we have $|M|=5$ if and only if the direct sum of all elements in $U$ has the form
				\begin{align*}
					L(g) \oplus \bigoplus_{f\in G\setminus\{g\}} L(f-\w) \oplus L(g-2\w),
				\end{align*}
				where $L$ is a line bundle  and $g\in G$.
				
			\end{remark}
			
			\begin{proposition} \label{Num rigid lemma}
				There are at most $5p+i-j+6$ line bundles in $\bigcap_{n=i,j}\Dom_{L}(E_n)$ such that their  direct sum is rigid.		
			\end{proposition}
			\begin{proof}  
				Note that $|I_1|=2p-j$ and $|I_2|=p+i+2$. Thus we have $$2p-j-(p+2)=p-j-2 \ \ \text{and} \ \ p+i+2-(p+2)=i.$$ 
				Since $|I_1|+|I_2|-3(p+i-j-2)=2(j-i+4)$, the set $\bigcap_{n=i,j}\Dom_{L}(E_n)$ consists of $p+i-j-2$  $(\w,2)$-segments and $j-i+4$ $(\w,1)$-segments by Lemma \ref{two or three elements}. By Lemma \ref{piece}, we have $5(p+i-j-2)+4(j-i+4)=5p+i-j+6$. 
			\end{proof}

				\begin{definition}Let $I \subset J \subset \L$. We say that $I$ is a \emph{$J$-upset}  if $(I+\L_{+}) \cap J \subset I$. 	Moreover,  an $J$-upset is called \emph{non-trivial} if  it is neither $J$ nor $\emptyset$. 
					For simplicity, we say that $I$ is an \emph{upset} of $\L$ if $J=\L$.
				\end{definition}

				Now we provide a program to construct a family of ACM tilting bundles on $\X$. 
				Our strategy consists of the following two steps.
				
			\emph{Step $1$}:	Let $i,j$ be integers with $0\le i\le j\le p-2$. For each $i\le n \le j$, let $E_n:=\langle n,k_n\rangle$ with  $E_s\in \Dom^{+}(E_t)$ for any $i\le s\le t \le j$, that is, the sequence $(k_n)_{i+1\le n \le j}$ satisfies $$k_n \le k_{n-1} \le k_{n}+1.$$

				  For $g, h\in G:=\{0,\x_1-\x_2,\x_1-\x_3,\x_2-\x_3\}$ with $g\neq h$, we set
					\begin{align*}
						V_{g,h}(E_i,E_j):=   \bigoplus_{{\x \in \mathcal{L}_{g}(E_i)\cup \mathcal{R}_{h}(E_j)}} \OO(\x),
					\end{align*}
					where  $\mathcal{L}_{g}(E_i)$ and $\mathcal{R}_{h}(E_j)$ are given in (\ref{the union of roofs}), that is,
					\begin{align*}
						\mathcal{L}_{g}(E_i)&=[(k_i-p-1)\x_4+g,(i+k_i)\x_4+g] \\
						\mathcal{R}_{h}(E_j)&=[\w+(j+k_j-p+1)\x_4+h,\w+(k_j+p)\x_4+h],
					\end{align*} and we enlarge the ranges of them by  $\mathcal{L}_{g}(E_0):=\emptyset$ and $R_{h}(E_{p-2}):=\emptyset$.	
				
				\emph{Step $2$}: Let					
				\begin{align} \label{H_{g,h}}
					H=H_{g,h}(E_i,E_j):= \bigcup_{q\in G'} \Big(\{\w+a\x_4+q\mid a\in J_1\} \cup \{b\x_4+q\mid b\in J_2\} \Big),
				\end{align}
				 where the integer intervals
				 \begin{align*}
				 	J_1:=[k_j-\delta_j,k_i-\delta^i]\ \text{ and }\ J_2:=[i+k_i+\delta^i-p-1,j+k_j+\delta_j-p-1]
				 \end{align*}
				   with $\delta^i:=1-\delta_{i,0}$ and $\delta_j:=\delta_{j,p-2}$ ($\delta_{ab}$ denotes the Kronecker delta, that is, $\delta_{ab}=1$ if $a=b$, otherwise $\delta_{ab}=0$) and 
					$G'$ is the subset of $G$ given by
					\begin{equation*}G':=
						\begin{cases} 
							G & \mbox{if $i= 0$, $j=p-2$,} \\
							G\setminus\{h\} & \mbox{if $i= 0$, $j\neq p-2$,} \\
							G\setminus\{g\} & \mbox{if $i\neq 0$, $j=p-2$,} \\
							G\setminus\{g,h\} & \mbox{if $i \neq 0$, $j\neq p-2$.} 			
						\end{cases}			
					\end{equation*}	
												
					For any $H$-upset $I$, we enlarge it to a subset 
					\begin{align*} \label{S_I}
						S_I:=I\cup (H-\w \setminus I-\w)\end{align*}
					of $U:=H \cup (H-\w)$. We set $S^{I}_{g,h}(E_i,E_j):=  \ \bigoplus_{\x\in S_I}\OO(\x)$ and
				\begin{equation} \label{form of tilting bundles}
					T:= \Big(\bigoplus\limits_{i\le n\le j}E_n\Big) \oplus 
					V_{g,h}(E_i,E_j)
					\oplus S^{I}_{g,h}(E_i,E_j).
				\end{equation}

				We give a typical example to explain our above construction.
				
				\begin{example} \label{example 1} Let $\X$ be a GL projective plane of type $(2,2,2,5)$. Let $E_1=\langle 1,0 \rangle$ and $E_2=\langle 2,0 \rangle$. Fix $g=\x_1-\x_2$ and $h=0$, the  domain $\bigcap_{n=1,2}\Dom_L (E_n)$ contains the line bundles over the left (resp. right) roof
				$$\mathcal{L}_{g}(E_1)=[-\x_1-\x_2-\x_4,\bar{x}_{12}+\x_4]\ \text{ (resp. } \mathcal{R}_{h}(E_2)=[\w-2\x_4,\w+\c]).$$	
				By (\ref{H_{g,h}}), $H=\{\bar{x}_{13}-4\x_4, \bar{x}_{23}-4\x_4\}$.
				We choose an $H$-upset $I:=\{\bar{x}_{13}-4\x_4\}$, and obtain the subset $S_I=\{\bar{x}_{13}-4\x_4, -\x_1+2\x_4\}$ of $U$ and so $T$. 				 
				 The  notation may be illustrated as in Figure \ref{trichotomy}, where we depict the indecomposable direct summands of $T$: the $2$-extension bundles $E_1$ and $E_2$ are shown in circles and  the line bundles contained in $V_{g,h}(E_1,E_2)$ (resp. $S^{I}_{g,h}(E_1,E_2)$)	are colored in red (resp. blue). 
			\begin{figure} [H]
					 \resizebox{\textwidth}{!}{
						\begin{xy} 0;<15pt,0pt>:<0pt,16pt>::
							(14,10) *+{(-\c-\x_4)} ="51",
							(14,6) *+{\langle 2,-4 \rangle} ="32",
							(16,8) *+{  \langle 3,-4 \rangle} ="42",
							(18,10) *+{ (-\c)} ="52",
							(14,2) *+{ \langle 0,-3 \rangle} ="13",
							(16,4) *+{ \langle 1,-3 \rangle} ="23",
							(18,6) *+{  \langle 2,-3 \rangle} ="33",
							(20,8) *+{ \langle 3,-3 \rangle} ="43",
							(22,10) *+{(-4\x_4)} ="53",
							(16,0) *+{ (\w-2\x_4)} ="04",
							(18,2) *+{  \langle 0,-2 \rangle} ="14",
							(20,4) *+{\langle 1,-2 \rangle} ="24",
							(22,6) *+{\langle 2,-2 \rangle} ="34",
							(24,8) *+{\langle 3,-2 \rangle } ="44",
							(26,10) *+{(-3\x_4)} ="54",
							(20,0) *+{ (\w-\x_4)} ="05",
							(22,2) *+{\langle 0,-1 \rangle} ="15",
							(24,4) *+{\langle 1,-1 \rangle} ="25",
							(26,6) *+{\langle 2,-1 \rangle} ="35",
							(28,8) *+{ \langle 3,-1 \rangle} ="45",
							(30,10) *+{(-2\x_4)} ="55",
							(30,10) *+{} ="5500",							
							(24,0) *+{(\w)} ="06",
							(26,2) *{ \langle 0,0 \rangle} ="16",
							(28,4) *+[Fo]{{\begin{smallmatrix} \scalebox{1}{$\langle 1,0 \rangle$}\end{smallmatrix}}} ="26",
							(30,6) *+[Fo]{{\begin{smallmatrix} \scalebox{1}{$\langle 2,0 \rangle$}\end{smallmatrix}}} ="36",
							(32,8) *+{ \langle 3,0 \rangle} ="46",
							(34,10) *+{(-\x_4)} ="56",
							(28,0) *+{ (\w+\x_4)} ="07",
							(28,0) *+{} ="0700",
							(30,2) *+{\langle 0,1 \rangle} ="17",
							(32,4) *+{\langle 1,1 \rangle} ="27",
							(34,6) *+{\langle 2,1 \rangle} ="37",
							(36,8) *+{\langle 3,1 \rangle} ="47",
							(38,10) *+{(0)} ="57",
							(32,0) *+{(\w+2\x_4)} ="08",
							(34,2) *+{\langle 0,2 \rangle} ="18",
							(36,4) *+{\langle 1,2 \rangle} ="28",
							(38,6) *+{ \langle 2,2 \rangle} ="38",
							(40,8) *+{ \langle 3,2 \rangle} ="48",
							(42,10) *+{(\x_4)} ="58",
							(36,0) *+{(\w+3\x_4)} ="09",
							(38,2) *+{ \langle 0,3 \rangle} ="19",
							(40,4) *+{\langle 1,3 \rangle} ="29",
							(42,6) *+{\langle 2,3 \rangle} ="39",
							(44,8) *+{\langle 3,3 \rangle} ="49",
							(40,0) *+{(\w+4\x_4)} ="010",
							(42,2) *+{\langle 0,4 \rangle} ="110",
							(44,4) *+{\langle 1,4 \rangle} ="210",
							(44,0) *+{(\w+\c)} ="011",
							(14,10.8) *+{\scalebox{1}{$-\c-\x_4$}},(14,11.6) *+{\scalebox{1}{$-\x_2-\x_3-\x_4$}},(14,12.4) *+{\scalebox{1}{$-\x_1-\x_3-\x_4$}},(14,13.2) *+{\red \scalebox{1.08}{$-\x_1-\x_2-\x_4$}},
							(18,10.8) *+{\scalebox{1}{$-\c$}},(18,11.6) *+{\scalebox{1}{$-\x_2-\x_3$}},(18,12.4) *+{\scalebox{1}{$-\x_1-\x_3$}},(18,13.2) *+{\red\scalebox{1.08}{$-\x_1-\x_2$}},
							(22,10.8) *+{\scalebox{1}{$-4\x_4$}},(22,11.6) *+{ \scalebox{1}{$\bar{x}_{23}-4\x_4$}},(22,12.4) *+{\blue \scalebox{1.08}{$\bar{x}_{13}-4\x_4$}},(22,13.2) *+{\red\scalebox{1.08}{$\bar{x}_{12}-4\x_4$}},
							(26,10.8) *+{\scalebox{1}{$-3\x_4$}},(26,11.6) *+{\red\scalebox{1.08}{$\bar{x}_{23}-3\x_4$}},(26,12.4) *+{\red\scalebox{1.08}{$\bar{x}_{13}-3\x_4$}},(26,13.2) *+{\red\scalebox{1.08}{$\bar{x}_{12}-3\x_4$}},
							(30,10.8) *+{\scalebox{1}{$-2\x_4$}},(30,11.6) *+{\red\scalebox{1.08}{$\bar{x}_{23}-2\x_4$}},(30,12.4) *+{\red\scalebox{1.08}{$\bar{x}_{13}-2\x_4$}},(30,13.2) *+{\red\scalebox{1.08}{$\bar{x}_{12}-2\x_4$}},
							(34,10.8) *+{\scalebox{1}{$-\x_4$}},(34,11.6) *+{\red\scalebox{1.08}{$\bar{x}_{23}-\x_4$}},(34,12.4) *+{\red\scalebox{1.08}{$\bar{x}_{13}-\x_4$}},(34,13.2) *+{\red\scalebox{1.08}{$\bar{x}_{12}-\x_4$}},
							(38,10.8) *+{\scalebox{1}{$0$}},(38,11.6) *+{\scalebox{1}{$\bar{x}_{23}$}},(38,12.4) *+{\scalebox{1}{$\bar{x}_{13}$}},(38,13.2) *+{\red\scalebox{1.08}{$\bar{x}_{12}$}},
							(42,10.8) *+{\scalebox{1}{$\x_4$}},(42,11.6) *+{\scalebox{1}{$\bar{x}_{23}+\x_4$}},(42,12.4) *+{\scalebox{1}{$\bar{x}_{13}+\x_4$}},(42,13.2) *+{\red\scalebox{1.08}{$\bar{x}_{12}+\x_4$}},			
							(16,-0.8) *+{\scalebox{1}{$-\x_3-3\x_4$}},(16,-1.6) *+{\scalebox{1}{$-\x_2-3\x_4$}},(16,-2.4) *+{\scalebox{1}{$-\x_1-3\x_4$}},(16,-3.2) *+{\red\scalebox{1.08}{$\w-2\x_4$}},
							(20,-0.8) *+{\scalebox{1}{$-\x_3-2\x_4$}},(20,-1.6) *+{\scalebox{1}{$-\x_2-2\x_4$}},(20,-2.4) *+{\scalebox{1}{$-\x_1-2\x_4$}},(20,-3.2) *+{\red\scalebox{1.08}{$\w-\x_4$}},
							(24,-0.8) *+{\scalebox{1}{$-\x_3-\x_4$}},(24,-1.6) *+{\red\scalebox{1.08}{$-\x_2-\x_4$}},(24,-2.4) *+{\red\scalebox{1.08}{$-\x_1-\x_4$}},(24,-3.2) *+{\red\scalebox{1.08}{$\w$}},
							(28,-0.8) *+{\scalebox{1}{$-\x_3$}},(28,-1.6) *+{\red\scalebox{1.08}{$-\x_2$}},(28,-2.4) *+{\red\scalebox{1.08}{$-\x_1$}},(28,-3.2) *+{\red\scalebox{1.08}{$\w+\x_4$}},
							(32,-0.8) *+{\scalebox{1}{$-\x_3+\x_4$}},(32,-1.6)*+{\red\scalebox{1.08}{$-\x_2+\x_4$}},(32,-2.4) *+{\red\scalebox{1.08}{$-\x_1+\x_4$}},(32,-3.2) *+{\red\scalebox{1.08}{$\w+2\x_4$}},
							(36,-0.8) *+{\scalebox{1}{$-\x_3+2\x_4$}},(36,-1.6) *+{\scalebox{1}{$-\x_2+2\x_4$}},(36,-2.4) *+{\color{blue}\scalebox{1.08}{$-\x_1+2\x_4$}},(36.3,-3.2) *+{\red\scalebox{1.08}{$\w+3\x_4$}},
							(40,-0.8) *+{\scalebox{1}{$-\x_3+3\x_4$}},(40,-1.6) *+{\scalebox{1}{$-\x_2+3\x_4$}},(40,-2.4) *+{\scalebox{1}{$-\x_1+3\x_4$}},(40,-3.2) *+{\red\scalebox{1.08}{$\w+4\x_4$}},
							(43.8,-0.8) *+{\scalebox{1}{$-\x_3+4\x_4$}},(43.8,-1.6) *+{\scalebox{1}{$-\x_2+4\x_4$}},(43.8,-2.4) *+{\scalebox{1}{$-\x_1+4\x_4$}},(44,-3.2) *+{\red\scalebox{1.08}{$\w+\c$}},
							"51", {\ar"42"},
							"32", {\ar"23"},
							"32", {\ar"42"},
							"42", {\ar"33"},
							"42", {\ar"52"},
							"52", {\ar"43"},
							"13", {\ar"23"},
							"33", {\ar"43"},
							"13", {\ar"04"},
							"23", {\ar"14"},
							"23", {\ar"33"},
							"33", {\ar"24"},
							"43", {\ar"34"},
							"43", {\ar"53"},
							"53", {\ar"44"},
							"04", {\ar"14"},
							"14", {\ar"24"},
							"34", {\ar"44"},
							"14", {\ar"05"},
							"24", {\ar"15"},
							"24", {\ar"34"},
							"34", {\ar"25"},
							"44", {\ar"35"},
							"44", {\ar"54"},
							"54", {\ar"45"},
							"05", {\ar"15"},
							"15", {\ar"25"},
							"35", {\ar"45"},
							"35", {\ar"26"},
							"15", {\ar"06"},
							"25", {\ar"16"},
							"25", {\ar"35"},
							"45", {\ar"36"},
							"45", {\ar"55"},
							"55", {\ar"46"},
							"06", {\ar"16"},
							"16", {\ar"26"},
							"36", {\ar"46"},
							"36", {\ar"27"},
							"16", {\ar"07"},
							"26", {\ar"17"},
							"26", {\ar"36"},
							"46", {\ar"37"},
							"46", {\ar"56"},
							"56", {\ar"47"},
							"07", {\ar"17"},
							"17", {\ar"27"},
							"37", {\ar"47"},
							"37", {\ar"28"},
							"27", {\ar"37"},
							"17", {\ar"08"},
							"27", {\ar"18"},
							"47", {\ar"38"},
							"47", {\ar"57"},
							"57", {\ar"48"},
							"08", {\ar"18"},
							"18", {\ar"28"},
							"28", {\ar"38"},
							"38", {\ar"48"},
							"38", {\ar"29"},
							"18", {\ar"09"},
							"28", {\ar"19"},
							"48", {\ar"39"},
							"48", {\ar"58"},
							"58", {\ar"49"},
							"09", {\ar"19"},
							"19", {\ar"29"},
							"39", {\ar"49"},
							"39", {\ar"210"},
							"19", {\ar"010"},
							"29", {\ar"110"},
							"29", {\ar"39"},
							"010", {\ar"110"},
							"110", {\ar"210"},
							"110", {\ar"011"},
							%
						\end{xy}
						} \caption{An example for the trichotomy of the object $T$} \label{trichotomy}
				\end{figure}
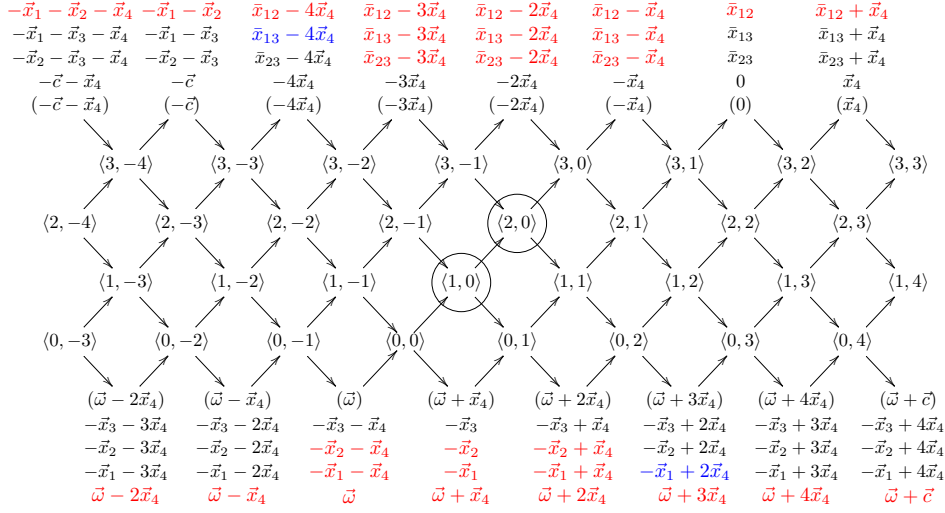					
				\end{example}

				By the definition of  $U$, we have
				\begin{eqnarray*}				
					U=\bigcup_{q\in G'} \Big(\{\w+a\x_4+q\mid a\in K_1\} \cup \{b\x_4+q\mid b\in K_2\} \Big)
				\end{eqnarray*}	
				with the integer intervals
				\begin{eqnarray*}
					K_1&=&[k_j-\delta_j,k_i-\delta^i]\cup [i+k_i+\delta^i+1,j+k_j+\delta_j+1], \\
					K_2&=&[i+k_i+\delta^i-p-1, j+k_j+\delta_j-p-1]\cup [k_j-\delta_j,k_i-\delta^i].
				\end{eqnarray*}	
				
				The following lemma explains the construction of $U$.
				
				\begin{lemma} \label{maximal subset}  Under the above setting, let $\Gamma$ be the maximal subset of 	
					$$\big\{\x\in\L \mid \OO(\x) \in \bigcap_{n=i,j}\Dom_{L}(E_n)\ \text{with}\ \x \notin \mathcal{L}_{g}(E_i)\cup \mathcal{R}_{h}(E_j)\big\} $$ such that $\OO(\x)\oplus\OO(\y)$ is rigid for any $\x \in \Gamma\ \text{and}\ \y \in \mathcal{L}_{g}(E_i)\cup \mathcal{R}_{h}(E_j)$. Then we have $U=\Gamma$.
				\end{lemma}	
				\begin{proof}
					
					Concerning the values of $i$ and $j$, we consider the following four cases.
					
					\emph{Case $1$}: For the case $i=0$ and $j=p-2$, we have $\mathcal{L}_{g}(E_i)= \mathcal{R}_{h}(E_j)=\emptyset$ and $G'=G$ by our convention. By (\ref{cap E_i,j}), $\x\in \Gamma$ if and only if 
					\begin{eqnarray*}
						\x=\w+a\x_4+f \ \ \text{or} \ \ b\x_4+f
					\end{eqnarray*}
					for some  $a\in [k_j-1,k_j+p]$, $b\in [k_i-p-1,k_i]$ and $f\in G$.
					
					\emph{Case $2$}: For the case $i=0$ and $j\neq p-2$, we have $\mathcal{L}_{g}(E_i)=\emptyset$.  Since $\mathcal{R}_{h}(E_j)$ is a convex set,    	$\OO(\x)\oplus\OO(\y)$ is rigid for all $\y \in \mathcal{R}_{h}(E_j)$,  which 
					is equivalent to 
					\begin{equation}\label{maximal subset 1}
						\w+(k_j-p)\x_4+h\le \x \le \w+(j+k_j+p+1)\x_4+h.
					\end{equation} 
					
					If moreover $\OO(\x) \in \bigcap_{n=i,j}\Dom_{\triangle,L}(E_n)$, 
					then 
					we have
					\begin{equation}\label{x_{0,notp-2}_1}
						\x=\w+a_1\x_4+h \ \ \text{or} \ \ \w+a_2\x_4+f,
					\end{equation}  
					where $a_1\in[j+k_j-p+1, k_j+p]$, $a_2\in[k_j, j+k_j+1]$ and $f\in G\setminus\{h\}$.
					
				On the other hand, if $\OO(\x) \in \bigcap_{n=i,j}\Dom_{\bigtriangledown,L}(E_n)$, we let $\x= c\x_4+f'$ with $c \in [k_i-p-1,k_i]$ and $f'\in G$. Since $E_i \in \Dom^{+}_\triangle(E_j)$, we have $k_j\le k_i\le j+k_j$, which implies that $\x$ satisfies (\ref{maximal subset 1})
					 if and only if $f'\neq h$.  
					Thus 
					\begin{equation}\label{x_{0,notp-2}_2}
						\x=c\x_4+f
					\end{equation}
					with $c \in [k_i-p-1,k_i]$ and $f\in G\setminus\{h\}$. By removing $x\in \mathcal{R}_{h}(E_j)$ from (\ref{x_{0,notp-2}_1}) and (\ref{x_{0,notp-2}_2}) respectively,  we further have that $\x\in \Gamma$ if and only if
					\begin{equation}\label{x_{0,notp-2}_Ans}
						\x=\w+a\x_4+f \ \ \text{or} \ \ b\x_4+f
					\end{equation}
					where $a\in [k_j,j+k_j+1]$, $b\in [k_i-p-1,j+k_j-p-1]\cup [k_j,k_i]$ and $f\in G\setminus\{h\}$.

					\emph{Case $3$}: For the case $i\neq 0$ and $j= p-2$, we have $\mathcal{R}_{h}(E_j)=\emptyset$. Since $\mathcal{L}_{g}(E_i)$ is a convex set, $\OO(\x)\oplus\OO(\y)$ is rigid for all $\y \in \mathcal{L}_{g}(E_i)$, which is equivalent to 
					\begin{equation*}
						(i+k_i-2p)\x_4+g	\le \x \le  (k_i+p-1)\x_4+g.
					\end{equation*} 
					One can easily check that if $\x\in \bigcap_{n=i,j}\Dom_{\bigtriangledown,L}(E_n)$, then we have
					\begin{equation}\label{x_{not0,p-2}_1}
						\x=c_1\x_4+g\ \ \text{or} \ \ c_2\x_4+f
					\end{equation}
					with  $c_1\in[k_i-p-1, i+k_i]$, $c_2\in [i+k_i-p, k_i-1]$ and $f\in G\setminus\{g\}$. 
					Otherwise $\x\in \bigcap_{n=i,j}\Dom_{\triangle,L}(E_n)$, by a similar argument as in Case $2$, we have
					\begin{equation}\label{x_{not0,p-2}_2}
						\x=\w+c\x_4+f
					\end{equation}
					with $c\in [k_j-1,k_j+p]$ and $f\in G\setminus\{g\}$. By removing $x\in R_{g}(E_i)$ from (\ref{x_{not0,p-2}_1}) and (\ref{x_{not0,p-2}_2}) respectively,   we further have that $\x\in \Gamma$ if and only if
					\begin{equation} \label{x_{not0,p-2}_Ans}
						\x=\w+a\x_4+f\ \ \text{or} \ \ b\x_4+f,
					\end{equation}
					where  $a\in [k_j-1,k_i-1]\cup[i+k_i+2,k_j+p]$, $b\in [i+k_i-p,k_i-1]$ and $f\in G\setminus\{g\}$.

					\emph{Case $4$}: For the case $i\neq 0$ and $j\neq p-2$, we assume $\OO(x) \in \bigcap_{n=i,j}\Dom_{L}(E_n)$ with $\x \notin \mathcal{L}_{g}(E_i)\cup \mathcal{R}_{h}(E_j)$.  
					Combining both Case $2$ and $3$,  $\OO(\x) \oplus \OO(\y)$ is rigid for any $\y \in \mathcal{L}_{g}(E_i)\cup \mathcal{R}_{h}(E_j)$ if and only if $\x$ satisfies both (\ref{x_{0,notp-2}_Ans}) and (\ref{x_{not0,p-2}_Ans}), that is, 
					\begin{equation*}
						\x=\w+a\x_4+f \ \ \text{or} \ \ b\x_4+f,
					\end{equation*}
					where 
					$a\in [k_j,k_i-1]\cup [i+k_i+2, j+k_j+1]$, $b\in [i+k_i-p,j+k_j-p-1]\cup [k_j,k_i-1]$ and $f\in G\setminus\{g,h\}$.

					On the other hand,  we can compute $U$  by substituting the value of $i$ and $j$ for the above four situations, respectively. The assertion follows immediately.
				\end{proof}
				

			Our main theorem in this section is the following.
			
			\begin{theorem}\label{main theorem}  		
				Let $V$ be an ACM bundle in $\coh \X$. Assume that any  rank-four indecomposable direct summands $E$ and $F$ of $V$ satisfy $E\in\Dom^{+}(F)$. Then $V$ is tilting on $\X$ if and only if $V$ has the form (\ref{form of tilting bundles}).
			\end{theorem}
							
				\begin{proof}[Proof of Theorem~\rm\ref{main theorem} `$\Rightarrow$'] Take a decomposition $V=\big(\bigoplus_{n\in J}E_n\big)\oplus P$, where $P$ is a maximal projective direct summand of $V$. Notice that for each rank-four indecomposable ACM bundle $E$, $\Dom^{+}(E)$ meets  the $\tau$-orbit of $E$ exactly once. Consequently, for any $s,t\in J$ with $s\neq t$, $E_s$ and $E_t$ belong to different $\tau$-orbits,  implying that  $J$ can be identified with a subset of $\{0,\dots,p-2\}$.  Let $i$ (resp. $j$) be the minimal (resp. maximal) element of $J$. Thus we have 
					\begin{equation} \label{number J}
						|J|\le j-i+1.
					\end{equation}					
					Let $E_n:=\langle n,k_n\rangle$ for $n\in J$.
					 By Lemma \ref{Domain property},  $\bigcap_{n\in J}\Dom_{L}(E_n)= \bigcap_{n=i,j}\Dom_{L}(E_n)$ holds. Moreover, by Proposition \ref{Num rigid lemma}, the rigidity of $V$ implies 
					\begin{equation} \label{number P}
						|P|\le 5p+i-j+6.
					\end{equation}
					It follows that $|V|=|P|+|J|\le 5p+7=|V|$, so both the $``\le"$ in (\ref{number J}) and (\ref{number P}) are  equalities. Then we have $J=[i,j]$ and thus $V$ contains the direct summands $\bigoplus_{i\le n\le j}E_n$, where $E_n=\langle n,k_n\rangle$ with $k_n \le k_{n-1} \le k_{n}+1$ for $i+1\le n \le j$.
					
					Next we show that there exist $g,h\in G$ with $g\neq h$ such that $V_{g,h}(E_i,E_j)$ is a direct summand of $V$.  We only show that there exists $g\in G$ such that $\OO(\x)$ is a direct summand of $V$ for any $\x\in \mathcal{L}_{g}(E_i)$ since one can show the other case similarly. If $i=0$, then this is nothing. Thus we assume $i\neq 0$. 
					By the proof of Proposition \ref{Num rigid lemma} and Remark \ref{double piece}, there are the following  direct summands of $V$
					\begin{equation}\label{equ 1.1}
						L_\ell(g_\ell) \oplus \bigoplus_{f\in G\setminus\{g_\ell\}} L_\ell(f-\w) \oplus L_\ell(g_\ell-2\w),
					\end{equation}
					where $L_\ell:=\OO((k_i+\ell-p-2)\x_4)$ and $g_\ell \in G$ for $1\le \ell \le i$. By Corollary \ref{rigid domain among line bundles cor}(c), the rigidity of $V$ implies $g_1=\dots=g_{i}=:g$. Since $V$ has the direct summand $L_1(g)$ (resp. $L_{i}(g-2\w$)) by (\ref{equ 1.1}),  we have that $L_1(g+\x_{l}-\w)$ (resp. $L_{i}(g-\w-\x_{l})$) is not a direct summand of $V$ for any $1\le l \le 3$ by Lemma \ref{rigid domain among line bundles}.  By cardinality reason, for any $1\le l \le 3$, there are also the following direct summands of $V$ 
								\begin{equation}\label{equ 1.2}
										L_1(g+\x_{l})\ \ \text{and} \ \ L_{i}(g-2\w-\x_{l}).
								\end{equation} 
								
						By Corollary \ref{rigid domain among line bundles cor}(b),  the direct summands $L_1(g)$ and $L_i(g-2\w)$ of $V$ implies that $V$ has no contain the line bundle $\OO(m\x_4+g-\w)$ for any $m\in \Z$. Moreover, by cardinality reason,  for any $a\in [k_i-p-1,i+k_i]$, $	\OO(a\x_4+g)$ is the direct summand of $V$.
						This together with (\ref{equ 1.1}) and 
								(\ref{equ 1.2}) implies that $\OO(\x)$ is a direct summand of $V$ for any $\x\in \mathcal{L}_{g}(E_i)$. In fact, $\x\in \mathcal{L}_{g}(E_i)$ if and only if $$\x=a\x_4+g \ \ \text{or}\ \ \x=\x_l+b\x_4+g$$ where $1\le l\le 3$, $a\in[k_i-p-1,i+k_i]$ and $b\in [k_i-p-1,i+k_i-p]$. 
					

					Finally we show that there exists an $H$-upset of $I$ such that  $S^{I}_{g,h}(E_i,E_j)$ is also a direct summand of $T$. Let $$\Sigma:=\{\x \in \L \mid \OO(\x) \in \add P, \ \x \notin \mathcal{L}_{g}(E_i)\cup \mathcal{R}_{h}(E_j)\}.$$
					By Lemma \ref{maximal subset},  $\Sigma$ is a subset of $U$, where $U:=H\cup (H-\w)$. By cardinality reason, each $\w$-orbit of line bundles in $U$ contains exactly one element in $\{\OO(\x) \mid \x \in \Sigma\}$. Let $I:=\Sigma \cap H$ and thus $\Sigma=I\cup (H-\w \setminus I-\w)=:S_I$. It remains to show that $I$ is an $H$-upset. Let $\x \in I$ and $\y \in H$ with $\x \le \y$. Assume $\y \notin \Sigma$, then we have $\y-\w \in \Sigma$. The rigidity of $V$ implies that $\y-\w-\x\le 2\c$, a contradiction since $\x \le \y$ holds. Hence $y\in \Sigma$ and then $y\in I$. Thus the assertion follows.
				\end{proof}
				
			In the rest, we prepare to prove the direction `$\Leftarrow$' of Theorem~\rm\ref{main theorem}.
			We need to prove that $T$ of the form $(\ref{form of tilting bundles})$ is a tilting bundle.
			To prove this, we need the following observations.

				\begin{lemma}\label{T is rigid} $V_{g,h}(E_i,E_j)$ is rigid.	
				\end{lemma}
				\begin{proof}  Let $\x,\y\in \mathcal{L}_{g}(E_i)\cup \mathcal{R}_{h}(E_j)$. Since $\mathcal{L}_{g}(E_i)$ and $\mathcal{R}_{h}(E_j)$ are both convex sets,
					by Lemma \ref{rigid domain among line bundles}, we only need show that the condition $-2\c \le\y-\x\le2\c$  holds for the four extreme cases, corresponding to the possible combinations of
					\begin{align*} 
						\x&=(k_i-p-1)\x_4+g \ \ \text{or} \ \  \w+(j+k_j-p+1)\x_4+h, \\ 
						\y&=(k_i+i)\x_4+g \ \ \text{or} \ \ \w+(k_j+p)\x_4+h.
					\end{align*} 
					We only prove the case when $\x=(k_i-p-1)\x_4+g$ and $\y=\w+(k_j+p)\x_4+h$ since the other cases can be shown similarly. Observe that $\w+h-g=-\x_{\ell}-\x_{4}$ holds for some $\ell$ with  $1\le \ell \le 3$. By a simple calculation, the inequality $-2\c \le\y-\x\le2\c$ is equivalent to $-2\c\le \x_{\ell}+(k_j-k_i)\x_4+\c \le 2\c$. Since $E_i \in \Dom^{+}_\triangle(E_j)$, we have $k_j\le k_i\le j+k_j$ and thus the inequality holds.
				\end{proof}
				
				\begin{lemma}\label{S is rigid} $S^{I}_{g,h}(E_i,E_j)$ is rigid.	
				\end{lemma}
				\begin{proof}
					Let $\OO(\x)$ and $ \OO(\y)$ be two indecomposable direct summands of $S^{I}_{g,h}(E_i,E_j)$. By Lemma \ref{rigid domain among line bundles}, it suffices to show that the condition $-2\c \le\y-\x\le2\c$ holds for any 
					$\x,\y \in S_I$. By definition $S_I=I\cup(H-\w\setminus I-\w)$,   the proof can be divided into the following three parts. 
					
					(1) We show that $-2\c \le\y-\x\le2\c$ holds for any $\x,\y \in H$.
					
					Fix $q_1,q_2\in G'$, it is enough to show that the condition $-2\c \le\y-\x\le2\c$ holds for the four extreme cases, corresponding to the possible combinations of
					\begin{align*} 
						\x&=\w+(k_i+i+\delta^i-p-1)\x_4+q_1 \ \ \text{or} \ \ (k_j-\delta_j)\x_4+q_1, \\
						\y&=\w+(k_j+j+\delta_j-p-1)\x_4+q_2 \ \ \text{or}\ \ (k_i-\delta^i)\x_4+q_2.
					\end{align*} 
					We only prove the case when $\x=\w+(k_i+i+\delta^i-p-1)\x_4+q_1$ and $\y=\w+(k_j+j+\delta_j-p-1)\x_4+q_2$ since the other assertions can be shown similarly.  By a simple calculation, the inequality $-2\c \le\y-\x\le2\c$ is equivalent to $$-2\c \le (k_j+j+\delta_j-k_i-i-\delta^i)\x_4+q_2-q_1 \le 2\c.$$ By the description of $E_i \in \Dom^{+}_\triangle(E_j)$, there exists an integer $0\le n_1 \le p-2$ such that $k_j+j-k_i-i=n_1$, which implies
					that $$-1\le k_j+j+\delta_j-k_i-i-\delta^i=n_1+\delta_{j,p-2}+\delta_{i,0}-1\le p-1.$$ Since $q_2-q_1\in \{0, \x_1-\x_2,\x_1-\x_3,\x_2-\x_3\}$, we have the assertion.
					
					(2) By (1), we have that $-2\c \le\y-\x\le2\c$ holds for any $\x,\y \in H-\w.$
					
					(3) 
					We show that $-2\c \le\y-\x-\w\le2\c$ holds for any $\x \in I$ and $\y \in H\setminus I$.

					Since $I$ is an $H$-upset and $\y \in H\setminus I$, we have $\y \not\ge \x$. Thus $\y-\x-\w \le 2\c$ holds by (\ref{two possibilities of x}). On the other hand, we assume $-2\c \not\le\y-\x-\w$. Again by (\ref{two possibilities of x}), we have $\x-\y \ge -2\w$. This would imply that $H$ has roofs, a contradiction.	
				\end{proof}

				Now we show that $T$ satisfies the rigid condition for being a tilting object.
				
				\begin{proposition} \label{rigid condition}
					We have $\Ext^{i}_{}(T,T)=0$ for any $i \neq 0$.
				\end{proposition}
				\begin{proof} (1) By our assumption, $k_n \le k_{n-1} \le k_{n}+1$ holds for $ i+1\le n \le j $,  implying that  $E_{s}\in\Dom_\triangle^{+}(E_{t})$ holds for any $i\le s\le t \le j$. Thus $\bigoplus_{i\le n\le j}E_n$ is rigid. 
					
					(2) By Lemma \ref{T is rigid} and \ref{S is rigid},  $V_{g,h}(E_i,E_j)$ and $S^{I}_{g,h}(E_i,E_j)$ are rigid.
					
					(3) 		
					By Lemma \ref{Domain property}, we have $\bigcap_{i\le n \le j}\Dom_{L}(E_n)= \bigcap_{n=i,j}\Dom_{L}(E_n)$. Thus $E_n\oplus \OO(\x)$ is rigid for any  $i\le n \le j$
					and $x\in \mathcal{L}_{g}(E_i)\cup \mathcal{R}_{h}(E_j)\cup S_{I}$. 
					
					(4)  By Lemma \ref{maximal subset}, we have $\OO(\y)\oplus\OO(\z)$ is rigid for any $\y \in U$ and $\z\in \mathcal{L}_{g}(E_i)\cup \mathcal{R}_{h}(E_j)$. Since $S_{I}\subset U$, then we have $V_{g,h}(E_i,E_j) \oplus S^{I}_{g,h}(E_i,E_j)$ is rigid.
					
					Combining all of them, we have the assertion. 			
				\end{proof}

				 First we prove the following special case of Theorem \ref{main theorem}.

					\begin{proposition} \label{special case}
					Assume $i=0$ and $j=p-2$. Then $T$ is $2$-tilting  on $\X$.				
				\end{proposition}
				\begin{proof} In this case,  $\{O(\x)\mid \x\in U\}=\bigcap_{n=0,p-2}\Dom_{L}(E_n)$ and  $T=E \oplus P$, where $E:=\bigoplus_{n=0}^{p-2}E_n$ and $P=S^{I}_{g,h}(E_0,E_{p-2}):=\bigoplus_{\x \in S_I}\OO(\x)$. Thanks to Theorem \ref{tilting-cluster tilting}, it suffices to show that $T$ is a slice in the $2$-cluster tilting subcategory
					$$\UU:=\add \{E(\ell \w), \,\OO(\x) \mid  \ell\in \Z,\, \x\in\L \}$$
					of $\vect \X$. Since $U$ has no roofs,  each $\w$-orbit of line bundles contains exactly two elements in $U$: One belongs to $H$, and the other belongs to $H-\w$. Then  $P$ meets each $\w$-orbit of line bundles exactly once, which implies that  $\UU=\add\{T(\ell \w)\mid \ell\in \Z\}$. Hence the condition (a) in Definition \ref{def.slice} is satisfied. 	Next we need to check condition (b) there, that is, $\Hom(T,T(\ell\w))=0$ for any $\ell >0$. 
					
					(1) We show that $\Hom_{}(\OO(\x),\OO(\y+\ell \w))=0$ for any  $\x, \y \in S_I$ and $\ell >0$.
					
					By (\ref{extension spaces between line bundles}), we have $\Hom_{}(\OO(\x),\OO(\y+\ell \w))=DR_{\y-\x+\ell \w}$. By (\ref{two possibilities of x}), this is zero if and only if $\y-\x+\ell \w \le 2\c+\w$ holds. Since $2\w<0$, it is enough to show that the inequality holds for $\ell=1,2$. By Lemma \ref{S is rigid}, the rigidity of $P$ implies that $\y-\x\le 2\c$. On the other hand, we have $\y-\x \not\ge -2\w$ since $S_I \subset U$. Therefore  $\y-\x+\w \le 2\c$ holds by (\ref{two possibilities of x}). 
					
					(2) We show that $\Hom_{}(E,\OO(\x+\ell\w))=0$ holds for any $\x\in S_I$ and $\ell >0$. 						 
					
					By Lemma \ref{Hom-domains 2}(a), it is easy to check that $\Hom_{}(E,\OO(\x))=0$ holds for any $\x \in H$. Since $S_I \subset  H \cup H-\w$, the assertion follows from Lemma \ref{Hom-domains}(a).

					(3) We show that $\Hom_{}(\OO(\x),E(\ell\w))=0$ holds for any $\x\in S_I$ and $\ell >0$.
					
					By Lemma \ref{Hom-domains 2}(b), it is easy to check that  $\Hom_{}(\OO(x),E)=0$ holds for any $\x \in H-\w$. Since $S_I \subset  H \cup H-\w$, the assertion follows from Lemma \ref{Hom-domains}(b).

					(4) We show that $\Hom_{}(E,E(\ell\w))=0$ holds for any $\ell >0$.  
					
					Since $E$ is a $2$-tilting object in $\underline{\ACM}\, \X$,  $\underline{\Hom}(E,E(\ell\w))=0$ holds for any $\ell>0$ by \cite[Proposition~2.3(b)]{HIO} and \cite[Proposition 4.28]{HIMO}. Note that for any $i\in\Z$ and  $\ell>0$, any composition $E \xrightarrow{f} \OO(\x+i\w) \xrightarrow{g} E(\ell\w)$ is zero. In fact, if $i>0$, then $f=0$ holds by (2), and if $i \le 0$, then $g=0$ holds by (3).
					Thus any morphism $E\to E(\ell\w)$ factors through $\lb \X$, and hence must be zero.
				\end{proof}
				
				Now we show that $T$ satisfies the remaining condition for being a tilting object.
				
				\begin{proposition} \label{gen. condition} We have $\thick T=\DDD^{\bo}(\coh\X)$.
				\end{proposition}
				\begin{proof}  We extend the direct summand $\bigoplus_{i\le k \le j}E_k$ of $T$ to $E:=\bigoplus_{k=0}^{p-2} E_k$, where $E_s:=\langle s, k_i\rangle$ and $E_t:=\langle t, k_j\rangle$ for  $0\le s \le i$ and $j\le t \le p-2$.
				
				As before,	let $l_m:=(k_i+m-p-2)\x_4$ and $r_{n}:=\w+(i+k_i+n-p)\x_4$ for $1\le m \le i$ and   $1\le n \le p-2-j$. 
					 We define $\mathcal{L}'_g(E_i)$ (resp. $\mathcal{R}'_h(E_j)$) as the set obtained from $\mathcal{L}_g(E_i)$ (resp. $\mathcal{R}_h(E_j)$) by replacing $$\bigcup_{m=1}^{i}\{l_m+g, l_m+g-2\w\} \  \big(\text{resp. } \bigcup_{n=1}^{p-2-j}\{r_n+h, r_n+h-2\w\}\big)$$ with  $$\{l_m(g-\w) \mid 1\le m \le i\} \ \ (\text{resp. } \{r_n(h-\w) \mid 1\le n \le p-2-j\}).$$
					 In this case, $P:=\mathcal{L}'_g(E_i)\cup \mathcal{R}'_h(E_j)\cup S_I$ forms a complete set of representation of $\L/\Z\w$ in $\L$. 
					 Put $W:=E\oplus (\bigoplus_{\x\in P}\OO(\x))$.  Clearly $\OO(\x)\in \Dom(E)$ for any $\x \in P$. 
					 Invoking Proposition \ref{rigid condition},  one can check that $W$ is a rigid object with $|W|=5p+7$. By the argument in the proof of Theorem \ref{main theorem} `$\Rightarrow$', $W$ has the form of (\ref{form of tilting bundles}), and further $W$ is tilting on $\X$ by Proposition \ref{special case}. Moreover, by Proposition \ref{reback lemma}, $W\in \thick T$ and hence we have $\thick W=\DDD^{\bo}(\coh\X)=\thick T$.	
				\end{proof}

				Now we are ready to complete the proof of Theorem \ref{main theorem}.
				
				\begin{proof}[Proof of Theorem~\rm\ref{main theorem} `$\Leftarrow$']
					Let $V$ be an ACM bundle of the form (\ref{form of tilting bundles}). It follows from Propositions \ref{rigid condition} and \ref{gen. condition} that $V$ is a tilting bundle on $\X$.
				\end{proof}	
				
			 	When $i=0$ and $j=p-2$, the direct summand $V_{g,h}(E_i,E_j)$ of $T$ is zero and thus
			 	 the trichotomy of $T$ reduces to two parts. We discuss the other case below.
				
				\begin{proposition} Let $T$ be an ACM bundle of the form (\ref{form of tilting bundles}). Then the direct summand  $S^{I}_{g,h}(E_i,E_j)$ of $T$ is zero if and only if  $i=j$ with $1\le j \le p-3$, that is,	there exists a $2$-extension bundle $E$ but not a $2$-Auslander bundle such that 
					\begin{align*}
						T=E\oplus \big(\bigoplus_{{\x \in \mathcal{L}_{g}(E)\cup R_{h}(E)}} \OO(\x)\big).
					\end{align*}
				\end{proposition}
				\begin{proof} Clearly $S^{I}_{g,h}(E_i,E_j)=\{0\}$ if and only if $H^{I}_{g,h}(E_i,E_j)=\emptyset$. By (\ref{H_{g,h}}), this is equivalent to 
					\begin{align} \label{11}
						k_i-\delta^{i}\le k_j-\delta_{j}-1 \ \text{ and }\ j+k_j+\delta_{j} \le i+k_i+\delta^{i}-1. 
					\end{align}
					By (\ref{Dom 2}), there exists integers $m,n\ge 0$ such that $i+m+n=j$ and $k_i-m=k_j$. 
					By a similar calculation, we obtain
					  $	\delta_{j}-\delta^{i}+1 \le m$ and $n\le -(\delta_{j}-\delta^{i}+1)$. Since $n\ge 0$, this implies that $\delta^{i}=1$, $\delta_{j}=0$ and $n=0$, that is, $i\neq 0$ and $j\neq p-2$. Moreover, we have $k_i\le k_j$ from (\ref{11}). Since $k_i-m=k_j$, we have $m\le 0$ and thus $m=0$. Therefore we have $i=j$ with $1\le j \le p-3$.
				\end{proof}

				\section{A classification theorem} \label{sec:A classification theorem}

				Let $U$ be a tilting object in $\underline{\ACM} \, \X$ such that $\underline{\End} (U)$ is hereditary, that is, 
							$U= \bigoplus_{n=0}^{p-2}E_n,$			
				where $E_n:=\langle n,k_n\rangle$ with the sequence $(k_n)$ satisfying $$k_n \le k_{n-1} \le k_{n}+1.$$				
				By Proposition \ref{special case}, $U$ can be completed to a $2$-tilting bundle $T$ on $\X$ of form (\ref{form of tilting bundles}) with $i=0$ and $j=p-2$.   It further follows from \cite[Theorems 7.15]{HIMO}
				    that $\vect \X$ has the $2$-cluster tilting subcategory  $$\mathcal{U}:=\add\{ T(\ell\w) \mid \ell\in \Z\} =\add\{ U(\ell\w),\,\OO(\x) \mid \ell\in \Z,\,\x\in\L \}.$$ 
				    		
				The aim of this section is to classify all ACM tilting bundles on $\X$ contained in the $2$-cluster tilting subcategory $\mathcal{U}$. We start with giving the follow observation.  
				  
				 	\begin{lemma}\label{at most one member}Let $E$ be an indecomposable ACM bundle of rank four. Then each tilting bundle on $\X$ contains at most one direct summand of the form $E(i\w)$, $i \in \Z$.
				 \end{lemma}
				 
				 \begin{proof} Assume  for contradiction that there exists a tilting bundle
				 	on $\X$ containing the direct summands  $E(n_1\w)$ and $E(n_2\w)$, where $n_1<n_2$. 
				 	By using Auslander-Reiten-Serre duality, we have
				 	$$	\Ext^2(E(n_1\w), E(n_2\w))\simeq D\Hom(E, E((n_1-n_2+1)\w)),$$	
				 	which is nonzero by Lemma \ref{Hom-domains}, a contradiction.
				 \end{proof} 
				  			
				Let $V\in \mathcal{U}$ be a tilting bundle  such that there are two rank-four indecomposable direct summands $E$ and $F$  satisfy $E\in \Dom^{-}(F)$. 
				Take a decomposition $$V=W\oplus P,$$ where $P$ is a maximal projective direct summand of $V$.  				
				Since $\mathcal{U}$ is closed under degree shift by $\Z\w$, without loss of generality,   we can
				assume that  $$W= M \oplus  N$$ for some $M  \in \add U$ and $N  \in \add U(\w)$.
				
				From these direct summands $\{E_n\mid 0\le n \le p-2\}$ of $U$, we choose a minimal index $i$ (resp. $s$) and a maximal $j$ (resp. $t$) such that $E_i$ and $E_j$ (resp. $E_s(\w)$ and $E_t(\w)$) belong to $\add M$ (resp. $\add N$). Observe that $\langle a,b \rangle(\w)=\langle p-2-a, a+b-p \rangle$ holds by Proposition \ref{L-action on 2-extension bundles}. Combining Lemma \ref{Domain property} and (\ref{cap E_i,j}), we have
				\begin{align*}
					D:=\Dom_L(W)&=\bigcap_{n=i,j}\Dom_L(E_n) \cap \bigcap_{n=s,t}\Dom_L(E_n(\w)) \\
					&=\bigcup_{n_1\in I_1}\langle -1,n_1 \rangle \cup \bigcup_{n_2\in I_2} \langle p-1,n_2+1 \rangle, \nonumber
				\end{align*} 
				where 
				\begin{align*}
					I_1:=[j+k_j-p+1,s+k_s]\ \text{ and }\ I_2:=[k_i-p-1,k_t-2].
				\end{align*}

				Immediately we have the following observations.

				\begin{lemma} \label{lemma do}
					We have
					\begin{align*}
						\{k_i+1, k_t-2\} \subset I_1 \ \text{ and }\ \{j+k_j-p+1, s+k_s-p-2 \} \subset I_2.
					\end{align*}
				\end{lemma}
				\begin{proof}
						We only show the first inclusion, as the second one can be shown similarly. Since $k_j\le k_i$ and $j\le p-2$, we have $j+k_j-p+1\le k_i+1$. By Proposition \ref{L-action on 2-extension bundles}, 
					$$E_s(\w)=\langle p-2-s, s+k_s-p \rangle \ \text{ and }\ E_i[-1]=\langle p-2-i, i+k_i-p+1 \rangle. $$
					The rigid property of the direct summands $E_i\oplus E_s(\w)$ of $U$ implies that $E_s(\w)\in \Dom^{+}(E_i[-1])$. Then we have $k_i-1\le k_s-2$ if $s\le i$, and $i+k_i-p+1 \le s+k_s-p$ otherwise. In both cases, we have $k_i+1 \le s+k_s$. Thus $k_i+1 \in I_1$. 	
					By a similar argument as above, one can check that $k_t-2 \in I_1$. 
				\end{proof}
				
				\begin{lemma} \label{D has no roofs}
					$D$ has no roofs, that is,  $|I_1|\le p+2$ and $|I_2|\le p+2$.
					
				\end{lemma}
				\begin{proof}
					First we show that either	$|I_1|\ge p+3$ or 	$|I_2|\ge p+3$  holds. Notice that  $|I_1|=s+k_s-j-k_j+p$ and $|I_2|=k_t-k_i+p$.	 Assume that both $s+k_s-j-k_j\ge 3$ and $k_t-k_i \ge 3$ hold. Thus we have $j<s$ and $t<i$ , a contradiction to  $i\le j$. 
					
					Next we assume that $|I_1|\ge p+3$. In this case, we have $s+k_s-j-k_j\ge 3$, which yields that $s \ge j$. Thus  
					\begin{align} \label{tsji}
						t\ge s \ge j \ge i.
					\end{align}
					This implies that $|W|\le (t-s+1)+(j-i+1)=t-s+j-i+2$. 
					
					On the other hand,	by a simple calculation,  $P$ consists of $|I_1|-p-2$  $(\w,2)$-segments, $|I_2|-(|I_1|-p-2)$ $(\w,1)$-segments and $k_i-k_t+2$ $(\w,0)$-segments. Then by Lemma \ref{piece}, we have 
					\begin{align*}
						|P|&\le 5(|I_1|-p-2)+4(|I_2|-(|I_1|-p-2)+(k_i-k_t+2))\\
						&\le 4p+6+s+k_s-j-k_j.
					\end{align*}
					Hence $|W|+|P|\le 4p+8+t-i+k_s-k_j$. Since by (\ref{tsji}) we obtain  $t-i\le p-2$ and $k_s\le k_j$, we conclude that $|V|=|W|+|P|\le 5p+6$, a contradiction to $|V|=5p+7$. 	
									
					The other case for $|I_2|\ge p+3$ can be shown similarly. 
				\end{proof}
								

				\begin{proposition}\label{Dom^-}We have $|P|=4p+8$. Moreover, 
					$V$ is a slice in the $2$-cluster tilting subcategory $\mathcal{U}$. 
				
				\end{proposition}	
					
					\begin{proof} By Lemma \ref{D has no roofs}, we have that each $\w$-orbit of line bundles contains either one or two elements in $D$. More precisely, there are precisely $|I_1|-(k_i-k_t+2)$	$(\w,1)$-segments. 
											Thus by Lemma \ref{piece} we have
											$$|P|\le 4(|I_1|+|I_2|)-4(|I_1|-(k_i-k_t+2))=4p+8.$$
											By Lemma \ref{at most one member}, we have $|W|\le p-1$. 
											It follows that $|V|=|W|+|P|\le 5p-7=|V|$. Therefore we have $|W|= p-1$ and $|P|=4p+8$, as claimed.

						By the above argument, each $\w$-orbit of line bundles contains exactly one element in $V$ and thus
					 we have $\mathcal{U}= \add \{V(\ell \w) \mid \ell \in \Z\}$. Hence the condition (a) in Definition \ref{def.slice} is satisfied. 	Next we need to check condition (b) there, that is, $\Hom(V,V(\ell\w))=0$ for any $\ell >0$. 					
					Let $W=M\oplus N:=\bigoplus_{k\in J}\langle a_k,b_k \rangle$.

					(1)  Let $\OO(\x), \OO(\y) \in \add P$. By (\ref{extension spaces between line bundles}),  we have $$\Hom_{}(\OO(\x),\OO(\y+\ell \w))=DR_{\y-\x+\ell \w}.$$ By (\ref{two possibilities of x}), this is zero if and only if $\y-\x+\ell \w \le 2\c+\w$ holds. Since $2\w<0$, it is enough to show that the inequality holds for $\ell=1,2$. The rigidity of $P$ implies that $\y-\x\le 2\c$. On the other hand, we have $\y-\x \not\ge -2\w$ since $\x,\y \in D$ and  $D$ has no roofs by Lemma \ref{D has no roofs}. Thus  $\y-\x+\w \le 2\c$ holds by (\ref{two possibilities of x}). Consequently we have $\Hom_{}(P,P(\ell \w))=0$ for any  $\ell >0$. 
					
					(2) We show that $\Hom(W, P(\ell\w))=0$  for any  $\ell >0$.
					
					We prove a slightly more general claim that $\Hom(W, \OO(\x+\ell\w))=0$ for any  $\x\in D$ and $\ell >0$. Thanks to Lemmas \ref{Hom-domains}(a), it suffices to show that 
					\begin{align} \label{112}
						\Hom(W, \langle -1,k_t-2\rangle )=0  \text{ and } \Hom(W, \langle p-1,(s+k_s-p-2)+1\rangle)=0.
					\end{align}							
					 Note that by Proposition \ref{L-action on 2-extension bundles}, we have 
					$$E_t(\w)=\langle p-2-t, t+k_t-p \rangle \ \text{ and }\ E_s(\w)=\langle p-2-s, s+k_s-p \rangle. $$ 
					 Then we have $a_k+b_k\ge k_t-2$ and $b_k\ge s+k_s-p$ for any $k\in J$.
					By Lemma \ref{Hom-domains 2}, we further have (\ref{112}).

					(3) We show that $\Hom(P, W(\ell\w))=0$ for any  $\ell >0$.
					
						We prove a slightly more general claim that $\Hom(\OO(\x-\ell\w), W )=0$ for any  $\x\in D$ and $\ell >0$. Thanks to Lemmas \ref{Hom-domains}(b), it suffices to show that 
					\begin{align} \label{113}
						\Hom( \langle -1,k_i+1\rangle,W )=0  \text{ and } \Hom( \langle p-1,(j+k_j-p+1)+1\rangle,W)=0.
					\end{align}	
					Since $k_i+1>k_i\ge b_k$ and $p-1+(j+k_j-p+1)+1>j+k_j\ge a_k+b_k$ hold for any $k\in J$, we  obtain (\ref{113}) by Lemma \ref{Hom-domains 2}.
					
					(4) We show that $\Hom(W, W(\ell\w))=0$ for any  $\ell >0$.
					
					By Proposition \ref{special case}, we have  $\Hom(U,U(\ell\w))=0$  for any $\ell >0$. It is enough to prove that $\Hom(N(-\w),M)=0$. Let $E_a:=\langle a,k_a \rangle\in \add N(-\w)$ and $E_b:=\langle b,k_b \rangle\in \add M$.  Then $E_a(\w)\in \Dom^{+}(E_b[-1])$ implies that $k_a \ge k_b+1$ if  $a\le b$, $a+k_a\ge b+k_b+1$ otherwise. In  both cases, we have $\Hom(E_a,E_b)=0$ by Lemma \ref{Hom-domains 2}. Thus $\Hom(W, W(\ell\w))=0$ for any  $\ell >0$.				This completes the proof.					
				\end{proof}

					Our main result in this section is the following.
				
				\begin{theorem} \label{A classification theorem}
					Let $V\in \mathcal{U}$. Then $V$ is tilting if and only if $V$ satisfies exactly one of the following conditions:
					\begin{itemize} 
						\item[(a)] $V$ is a $2$-canonical tilting bundle, up to degree shift.
						\item[(b)] $V$ is a slice in the $2$-cluster tilting subcategory $\mathcal{U}$.
						\item[(c)] $V$ has the form (\ref{form of tilting bundles})  with $i\neq 0$ or $j\neq p-2$.
					\end{itemize}
				\end{theorem}

				\begin{proof}
					`$\Leftarrow$': This  follows directly from \cite[Theorem 6.2]{HIMO}, Theorems \ref{tilting-cluster tilting} and \ref{main theorem}. 
					
					`$\Rightarrow$': We divide the proof into the following two cases.
					
					\emph{Case $1$}: $V$ consists  of line bundles. By Theorem \ref{ACM tilting bundles consisting of line bundles},  the statement (a) holds.

					\emph{Case $2$}: $V$ contains indecomposable direct summands of rank four. Take a decomposition $V=W\oplus P$, where $P$ is a maximal projective direct summand of $V$. 
					
					Assume that for any indecomposable direct summands $E$ and $F$ of $W$, we have
					$E\in\Dom^{+}(F)$. Then by Theorem \ref{main theorem}, $U$ has the form (\ref{form of tilting bundles}). Note that  $|W|\le p-1$ holds by Lemma \ref{at most one member}. If $|W|\le p-2$, then the  statement (c) holds. Otherwise by Proposition \ref{special case} and Theorem \ref{tilting-cluster tilting}, the  statement (b) holds. 
					
					Assume that there exist indecomposable direct summands $E$ and $F$ of $W$ such that
					$E\in\Dom^{-}(F)$. Then the statement (b) holds by Proposition \ref{Dom^-}. 
				\end{proof}

				As a direct consequence of the preceding result, we have the following corollary.  
				
				\begin{corollary}Let $V\in \mathcal{U}$ be a tilting bundle on $\X$. Take a decomposition $V=W\oplus P$, where $P$ is a maximal projective direct summand of $V$. Then
					\begin{itemize} 
						\item[(a)]   $V$ is a $2$-canonical tilting bundle (up to degree shift) if and only if $|W|=0$.
						\item[(b)] $V$ has the form (\ref{form of tilting bundles})  with $i\neq 0$ or $j\neq p-2$ if and only if $1\le |W| \le p-2$. 
						\item[(c)]   $V$ is a slice in the $2$-cluster tilting subcategory $\mathcal{U}$ if and only if $|W|=p-1$.
					\end{itemize}
				\end{corollary}

				\begin{remark} We remark that for a general $2$-cluster tilting subcategory $\MM$ of $\vect \X$, there may exist a tilting bundle on $\X$ contained in $\MM$ different from the cases in Theorem \ref{A classification theorem}. 
				 For example, let $\X$ be of weight type $(2,2,2,5)$. As in Setting \ref{setting 1}, we let $W:=\langle 3,-2 \rangle \oplus \langle 1,0 \rangle \oplus \langle 3,0 \rangle $ and $M:=W\oplus \langle 0,0\rangle$. 	
				 			 
				 It is easy to check that $M$ is a $2$-tilting object in $\underline{\ACM} \, \X$, and  it further induces the cluster tilting category $\MM:=\{M(\ell\w),\, \OO(\x)\mid \ell\in \Z,\, \x\in \L\}$ of $\vect \X$ by \cite[Theorems 4.53, 5.23]{HIMO}. Let $g:=\x_1-\x_2$ and $S:=\bigcup_{f\in G\setminus\{g\}}[-4\x_4+f,-\x_4+f]$.
				We claim that the object $$V:=W\oplus \big(\bigoplus_{\x\in S\cup\mathcal{L}_{g}(\langle 1,0 \rangle)}\OO(\x)\big)$$
				is a tilting bundle on $\X$, where $\mathcal{L}_{g}(\langle 1,0 \rangle):=[-\x_1-\x_2-\x_4,\bar{x}_{12}+\x_4]$. 
				It can be described as in Figure \ref{An remark},	where we depict the indecomposable direct summands of $V$: the $2$-extension bundles $\langle 3,-2\rangle$, $\langle 1,0\rangle$ and $\langle 3,0\rangle$ are shown in circles, and the line bundles over $\mathcal{L}_{g}(\langle 1,0 \rangle)$ (resp. $S$) are colored in red (resp. blue).		
				\begin{figure}[H] \resizebox{\textwidth}{!}{
					\begin{xy} 0;<15pt,0pt>:<0pt,16pt>::
						(14,10) *+{(-\c-\x_4)} ="51",
						(14,6) *+{\langle 2,-4 \rangle} ="32",
						(16,8) *+{  \langle 3,-4 \rangle} ="42",
						(18,10) *+{ (-\c)} ="52",
						(14,2) *+{ \langle 0,-3 \rangle} ="13",
						(16,4) *+{ \langle 1,-3 \rangle} ="23",
						(18,6) *+{  \langle 2,-3 \rangle} ="33",
						(20,8) *+{ \langle 3,-3 \rangle} ="43",
						(22,10) *+{(-4\x_4)} ="53",
						(16,0) *+{ (\w-2\x_4)} ="04",
						(18,2) *+{  \langle 0,-2 \rangle} ="14",
						(20,4) *+{\langle 1,-2 \rangle} ="24",
						(22,6) *+{\langle 2,-2 \rangle} ="34",
						(24,8) *+[Fo]{{\begin{smallmatrix} \scalebox{1}{$\langle 3,-2 \rangle$}\end{smallmatrix}}} ="44",
						(26,10) *+{(-3\x_4)} ="54",
						(20,0) *+{ (\w-\x_4)} ="05",
						(22,2) *+{\langle 0,-1 \rangle} ="15",
						(24,4) *+{\langle 1,-1 \rangle} ="25",
						(26,6) *+{\langle 2,-1 \rangle} ="35",
						(28,8) *+{ \langle 3,-1 \rangle} ="45",
						(30,10) *+{(-2\x_4)} ="55",
						(30,10) *+{} ="5500",							
						(24,0) *+{(\w)} ="06",
						(26,2) *+{\langle 0,0 \rangle} ="16",
						(28,4) *+[Fo]{{\begin{smallmatrix} \scalebox{1}{$\langle 1,0 \rangle$}\end{smallmatrix}}} ="26",
						(30,6) *+{\ \scalebox{1}{${\langle 2,0 \rangle}$}} ="36",
						(32,8) *+[Fo]{{\begin{smallmatrix} \scalebox{1}{$\langle 3,0 \rangle$}\end{smallmatrix}}} ="46",
						(34,10) *+{(-\x_4)} ="56",
						(28,0) *+{ (\w+\x_4)} ="07",
						(28,0) *+{} ="0700",
						(30,2) *+{\langle 0,1 \rangle} ="17",
						(32,4) *+{\langle 1,1 \rangle} ="27",
						(34,6) *+{\langle 2,1 \rangle} ="37",
						(36,8) *+{\langle 3,1 \rangle} ="47",
						(38,10) *+{(0)} ="57",
						(32,0) *+{(\w+2\x_4)} ="08",
						(34,2) *+{\langle 0,2 \rangle} ="18",
						(36,4) *+{\langle 1,2 \rangle} ="28",
						(38,6) *+{ \langle 2,2 \rangle} ="38",
						(40,8) *+{ \langle 3,2 \rangle} ="48",
						(42,10) *+{(\x_4)} ="58",
						(36,0) *+{(\w+3\x_4)} ="09",
						(38,2) *+{ \langle 0,3 \rangle} ="19",
						(40,4) *+{\langle 1,3 \rangle} ="29",
						(42,6) *+{\langle 2,3 \rangle} ="39",
						(44,8) *+{\langle 3,3 \rangle} ="49",
						(40,0) *+{(\w+4\x_4)} ="010",
						(42,2) *+{\langle 0,4 \rangle} ="110",
						(44,4) *+{\langle 1,4 \rangle} ="210",
						(44,0) *+{(\w+\c)} ="011",
						(14,10.8) *+{\scalebox{1}{$-\c-\x_4$}},(14,11.6) *+{\scalebox{1}{$-\x_2-\x_3-\x_4$}},(14,12.4) *+{\scalebox{1}{$-\x_1-\x_3-\x_4$}},(14,13.2) *+{\red \scalebox{1.08}{$-\x_1-\x_2-\x_4$}},
						(18,10.8) *+{\scalebox{1}{$-\c$}},(18,11.6) *+{\scalebox{1}{$-\x_2-\x_3$}},(18,12.4) *+{\scalebox{1}{$-\x_1-\x_3$}},(18,13.2) *+{\red\scalebox{1.08}{$-\x_1-\x_2$}},
						(22,10.8) *+{\blue \scalebox{1}{$-4\x_4$}},(22,11.6) *+{\blue \scalebox{1.08}{$\bar{x}_{23}-4\x_4$}},(22,12.4) *+{\blue \scalebox{1.08}{$\bar{x}_{13}-4\x_4$}},(22,13.2) *+{\red\scalebox{1.08}{$\bar{x}_{12}-4\x_4$}},
						(26,10.8) *+{\blue\scalebox{1}{$-3\x_4$}},(26,11.6) *+{\blue\scalebox{1.08}{$\bar{x}_{23}-3\x_4$}},(26,12.4) *+{\blue\scalebox{1.08}{$\bar{x}_{13}-3\x_4$}},(26,13.2) *+{\red\scalebox{1.08}{$\bar{x}_{12}-3\x_4$}},
						(30,10.8) *+{\blue\scalebox{1}{$-2\x_4$}},(30,11.6) *+{\blue\scalebox{1.08}{$\bar{x}_{23}-2\x_4$}},(30,12.4) *+{\blue\scalebox{1.08}{$\bar{x}_{13}-2\x_4$}},(30,13.2) *+{\red\scalebox{1.08}{$\bar{x}_{12}-2\x_4$}},
						(34,10.8) *+{\blue \scalebox{1}{$-\x_4$}},(34,11.6) *+{\blue\scalebox{1.08}{$\bar{x}_{23}-\x_4$}},(34,12.4) *+{\blue\scalebox{1.08}{$\bar{x}_{13}-\x_4$}},(34,13.2) *+{\red\scalebox{1.08}{$\bar{x}_{12}-\x_4$}},
						(38,10.8) *+{\scalebox{1}{$0$}},(38,11.6) *+{\scalebox{1}{$\bar{x}_{23}$}},(38,12.4) *+{\scalebox{1}{$\bar{x}_{13}$}},(38,13.2) *+{\red\scalebox{1.08}{$\bar{x}_{12}$}},
						(42,10.8) *+{\scalebox{1}{$\x_4$}},(42,11.6) *+{\scalebox{1}{$\bar{x}_{23}+\x_4$}},(42,12.4) *+{\scalebox{1}{$\bar{x}_{13}+\x_4$}},(42,13.2) *+{\red\scalebox{1.08}{$\bar{x}_{12}+\x_4$}},			
						(16,-0.8) *+{\scalebox{1}{$-\x_3-3\x_4$}},(16,-1.6) *+{\scalebox{1}{$-\x_2-3\x_4$}},(16,-2.4) *+{\scalebox{1}{$-\x_1-3\x_4$}},(16,-3.2) *+{\scalebox{1}{$\w-2\x_4$}},
						(20,-0.8) *+{\scalebox{1}{$-\x_3-2\x_4$}},(20,-1.6) *+{\scalebox{1}{$-\x_2-2\x_4$}},(20,-2.4) *+{\scalebox{1}{$-\x_1-2\x_4$}},(20,-3.2) *+{\scalebox{1}{$\w-\x_4$}},
						(24,-0.8) *+{\scalebox{1}{$-\x_3-\x_4$}},(24,-1.6) *+{\red\scalebox{1.08}{$-\x_2-\x_4$}},(24,-2.4) *+{\red\scalebox{1.08}{$-\x_1-\x_4$}},(24,-3.2) *+{\red\scalebox{1.08}{$\w$}},
						(28,-0.8) *+{\scalebox{1}{$-\x_3$}},(28,-1.6) *+{\red\scalebox{1.08}{$-\x_2$}},(28,-2.4) *+{\red\scalebox{1.08}{$-\x_1$}},(28,-3.2) *+{\red\scalebox{1.08}{$\w+\x_4$}},
						(32,-0.8) *+{\scalebox{1}{$-\x_3+\x_4$}},(32,-1.6)*+{\red\scalebox{1.08}{$-\x_2+\x_4$}},(32,-2.4) *+{\red\scalebox{1.08}{$-\x_1+\x_4$}},(32,-3.2) *+{\red\scalebox{1.08}{$\w+2\x_4$}},
						(36,-0.8) *+{\scalebox{1}{$-\x_3+2\x_4$}},(36,-1.6) *+{\scalebox{1}{$-\x_2+2\x_4$}},(36,-2.4) *+{\scalebox{1}{$-\x_1+2\x_4$}},(36.3,-3.2) *+{\scalebox{1}{$\w+3\x_4$}},
						(40,-0.8) *+{\scalebox{1}{$-\x_3+3\x_4$}},(40,-1.6) *+{\scalebox{1}{$-\x_2+3\x_4$}},(40,-2.4) *+{\scalebox{1}{$-\x_1+3\x_4$}},(40,-3.2) *+{\scalebox{1}{$\w+4\x_4$}},
						(43.8,-0.8) *+{\scalebox{1}{$-\x_3+4\x_4$}},(43.8,-1.6) *+{\scalebox{1}{$-\x_2+4\x_4$}},(43.8,-2.4) *+{\scalebox{1}{$-\x_1+4\x_4$}},(44,-3.2) *+{\scalebox{1}{$\w+\c$}},
						"51", {\ar"42"},
						"32", {\ar"23"},
						"32", {\ar"42"},
						"42", {\ar"33"},
						"42", {\ar"52"},
						"52", {\ar"43"},
						"13", {\ar"23"},
						"33", {\ar"43"},
						"13", {\ar"04"},
						"23", {\ar"14"},
						"23", {\ar"33"},
						"33", {\ar"24"},
						"43", {\ar"34"},
						"43", {\ar"53"},
						"53", {\ar"44"},
						"04", {\ar"14"},
						"14", {\ar"24"},
						"34", {\ar"44"},
						"14", {\ar"05"},
						"24", {\ar"15"},
						"24", {\ar"34"},
						"34", {\ar"25"},
						"44", {\ar"35"},
						"44", {\ar"54"},
						"54", {\ar"45"},
						"05", {\ar"15"},
						"15", {\ar"25"},
						"35", {\ar"45"},
						"35", {\ar"26"},
						"15", {\ar"06"},
						"25", {\ar"16"},
						"25", {\ar"35"},
						"45", {\ar"36"},
						"45", {\ar"55"},
						"55", {\ar"46"},
						"06", {\ar"16"},
						"16", {\ar"26"},
						"36", {\ar"46"},
						"36", {\ar"27"},
						"16", {\ar"07"},
						"26", {\ar"17"},
						"26", {\ar"36"},
						"46", {\ar"37"},
						"46", {\ar"56"},
						"56", {\ar"47"},
						"07", {\ar"17"},
						"17", {\ar"27"},
						"37", {\ar"47"},
						"37", {\ar"28"},
						"27", {\ar"37"},
						"17", {\ar"08"},
						"27", {\ar"18"},
						"47", {\ar"38"},
						"47", {\ar"57"},
						"57", {\ar"48"},
						"08", {\ar"18"},
						"18", {\ar"28"},
						"28", {\ar"38"},
						"38", {\ar"48"},
						"38", {\ar"29"},
						"18", {\ar"09"},
						"28", {\ar"19"},
						"48", {\ar"39"},
						"48", {\ar"58"},
						"58", {\ar"49"},
						"09", {\ar"19"},
						"19", {\ar"29"},
						"39", {\ar"49"},
						"39", {\ar"210"},
						"19", {\ar"010"},
						"29", {\ar"110"},
						"29", {\ar"39"},
						"010", {\ar"110"},
						"110", {\ar"210"},
						"110", {\ar"011"},
						%
					\end{xy}
				} \caption{A tilting bundle beyond the cases in Theorem \ref{A classification theorem}} \label{An remark}
			\end{figure}

			The rigidity of $V$ follows from Proposition \ref{Rigid-domains} and Corollary \ref{rigid domain among line bundles cor}. Moreover,  the exactness of the sequence $$0\to \langle 3,-2\rangle \to \big(\bigoplus_{f\in G} \OO(-3\x_4+f)\big)\oplus \langle 1,0\rangle \to \langle 2,0\rangle \to 0$$ in $\coh \X$ implies that $\langle 2,0\rangle\in \thick W$.  Since  $\Dom^{+}_{\triangle}(\langle 1,0\rangle)\subset \thick W$ holds by Proposition \ref{reback lemma}, we have that $\thick W$ contains the slice 
				$$T=\big(\bigoplus_{0\le i\le 3}\langle i,0\rangle \big)\oplus \big(\bigoplus_{\x\in S'} \OO(\x)  \big)$$				
				in $\MM$, where $S':=\bigcup_{0\le n_1\le 2}\langle -1,n_1 \rangle \cup \bigcup_{-4\le n_2\le -1}\langle p-1,n_2+1 \rangle$. Hence we have $\thick W=\DDD^{\bo}(\coh\X)$ and  the claim follows. 		
				\end{remark}

				\section{Endomorphism algebras of $T$}	 \label{sec: Endomorphism algebras of}
				In this section, we show that the endomorphism algebras of the form {(\ref{form of tilting bundles})} are  almost $2$-representation infinite algebras.  Let us start with recalling the following basic concepts in \cite{HIMO, HIO}.
				
				Let $\Lambda$ be a finite dimensional $\bf k$-algebra of finite global dimension. The \emph{$2$-shifted Nakayama functors} are defined by	
				\begin{align*}
					\nu_2&:=(D\Lambda)[-2]\Lotimes_\Lambda-: \DDD^{\bo}(\mod\Lambda)\to \DDD^{\bo}(\mod\Lambda),\\
					\nu_2^{-1}&:=\RHom_{\Lambda}(D\Lambda,-[2]):\DDD^{\bo}(\mod\Lambda)\to\DDD^{\bo}(\mod\Lambda).
				\end{align*}	
				They are mutually quasi-inverse. Moreover, for any $X,Y\in\DDD^{\bo}(\mod\Lambda)$, there exists the following functorial isomorphism
				$$\Hom_{\DDD^{\bo}(\mod\Lambda)}(X,Y)\simeq D\Hom_{\DDD^{\bo}(\mod\Lambda)}(Y,\nu_2 X[2]).$$
				
				The following classes of finite dimensional algebras are central in this section.
				\begin{definition} (See \cite[Definition 2.7]{HIO}, \cite[Definition 2.21]{HIMO}) \label{define almost d-RI} 
					Let $\Lambda$ be a finite dimensional $\bf k$-algebra of finite global dimension.					
					\begin{itemize}
					\item[(a)] The algebra 	$\Lambda$ is called \emph{$2$-representation infinite}	if 	$\nu_2^{-i}(\Lambda)\in\mod \Lambda$ for all $i\ge0$. This is equivalent to that $\nu_2^i(D\Lambda)\in\mod\Lambda$ for all
					$i\ge0$ by \cite[Proposition 2.9]{HIO}.		
					\item[(b)] The algebra $\Lambda$ is called \emph{almost $2$-representation infinite} if
					$H^j(\nu_2^{-i}(\Lambda))=0$ holds for all $i\in\Z$ and all $j\in\Z\setminus\{0,2\}$.
					\end{itemize}		
				\end{definition}
				
				Clearly any $2$-representation infinite algebra is almost $2$-representation infinite.  Moreover, the following properties are shown in \cite[Proposition 2.22]{HIMO}.
				
				\begin{proposition} \label{2-IF and 2-AIF} 
					\begin{itemize}
						\item[(a)] $2$-representation infinite algebras are precisely
						almost $2$-representation infinite algebras of global dimension $2$
						\item[(b)] An almost $2$-representation infinite algebra has global dimension $2$ or $4$.
					\end{itemize}
				\end{proposition}
				
				Let $V$ be a tilting bundle on $\X$ with $\Lambda:=\End_{}(V).$ 
				According to \cite{HIMO},  there are triangle equivalences
				\[V\Lotimes_\Lambda-:\DDD^{\bo}(\mod\Lambda)\to\DDD^{\bo}(\coh\X)
				\ \mbox{ and }\ \RHom_{\X}(V,-):\DDD^{\bo}(\coh\X)\to\DDD^{\bo}(\mod\Lambda),\]
				which are mutually quasi-inverse and make the following diagram commutative.
				\begin{equation}\label{derived equivalence}
					\xymatrix{
						\DDD^{\bo}(\mod\Lambda)\ar@{-}[rr]^{\sim}\ar[d]^{\nu_2}&&\DDD^{\bo}(\coh\X)\ar[d]^{(\w)}\\
						\DDD^{\bo}(\mod\Lambda)\ar@{-}[rr]^{\sim}&&\DDD^{\bo}(\coh\X)
				}\end{equation}
				
				Recall from \cite{HIMO} that the endomorphism algebra of the $2$-canonical titling bundle $T^{\ca}:=\bigoplus_{0\le \x \le 2\c}\OO(\x)$ is an almost $2$-representation infinite algebra.
				The following result shows that the same holds for the endomorphism algebra of $T$ given in {(\ref{form of tilting bundles})}.

				\begin{theorem} \label{endomorphism algebra of T}
					Let $T:=T_{g,h}(E_i,E_j)$ be a tilting bundle on $\X$ of the form {(\ref{form of tilting bundles})} and $\Lambda:=\End_{}(T)$. Then we have the following assertions.
					\begin{itemize} 
						\item[(a)] $\Lambda$ is an almost $2$-representation infinite algebra.
						\item[(b)] $\Lambda$ is a $2$-representation infinite algebra if and only if $i=0$ and $j=p-2$.
					\end{itemize}
				\end{theorem}
				\begin{proof}(a) The algebra $\Lambda$ is derived equivalent to the $2$-canonical algebra $T^{\ca}$, which has finite global dimension. Thus $\Lambda$ has finite global dimension. Moreover,  we have
					\begin{align*}
						H^j(\nu_2^{-i}(\Lambda))&\simeq 
						\Hom_{\DDD^{\bo}(\mod\Lambda)}(\Lambda,\nu_2^{-i}(\Lambda)[j]) \\
						&\simeq\Hom_{\DDD^{\bo}(\coh\X)}(T,T(-i\w)[j]) \ \ \ \ \ \   (\text{by (\ref{derived equivalence})} )\\
						&\simeq\Ext_{}^{j}(T,T(-i\w)) \ \ \ \ \ \  \ \ \ \ \ \ \ \  \ \ \  \ \ (T \in \coh \X )
					\end{align*}
					for any $i,j\in \Z$. 
					This is clearly zero for $j<0$. Since $\coh \X$ has global dimension $2$, this is zero for $j>2$. It remains to show that $H^1(\nu_2^{-i}(\Lambda))=0$ holds for any $i\in \Z$. Take a decomposition $T=E \oplus P$, where $P$ is a maximal projective direct summand of $T$. By (\ref{ACM bundle eq.}) and (\ref{extension spaces between line bundles}), we only need to prove that $\Ext^{1}(E,E(\ell\w))=0$ holds for any $\ell \in \Z$. 					
					Using Auslander-Reiten-Serre duality,  we can reduce our consideration to the case $\ell \ge 0$. 
					
					Note that $E(\x_3)\simeq E[1]$. 
					By applying $\Hom(E(-\ell\w),-)$ to the exact sequence $0\to E \to \mathfrak{I}(E) \xrightarrow{\phi} E(\x_3) \to 0$, and using a parallel argument as in the proof of Proposition \ref{Rigid-domains}, one can show that 
					\begin{align} \label{2-IF and 2-AIF ineq.}
						\Ext^{1}(E,E(\ell\w)) \simeq \underline{\Hom}(E,E(\x_3+\ell\w)).
					\end{align}		
								
					For  $\ell=0$, it follows directly from the rigidity of $E$ that
					(\ref{2-IF and 2-AIF ineq.}) equals zero. 
					
					For  $\ell>0$, we note that  for any $s,t$ with $i\le s\le t \le j$, $E_s \in \Dom^{+}(E_t)$ implies that $\Hom(E_s,E_t(-\x_4))=0$ by definition. Then by Proposition \ref{L-action on 2-extension bundles} we have $$\Hom(E,E(\x_3+\w))\simeq \Hom(E,E(-\x_4))=0.$$  It further follows from Lemma \ref{Hom-domains}(a) that $\Hom(E,E(\x_3+\ell\w))=0$ holds for $\ell>0$. Thus (\ref{2-IF and 2-AIF ineq.}) equals zero for $\ell>0$.  
					Hence the assertion follows.
					
					(b) Assume $i=0$ and $j=p-2$. By Proposition \ref{special case}, it follows that $T$ is a  $2$-tilting bundle on $\X$.  Moreover,  by \cite[Proposition 7.14]{HIMO}, we have that $\Lambda$ is a $2$-representation infinite algebra.
					
					Conversely, we assume for contradiction that $i\neq 0$ or $j\neq p-2$. In this case, we have $\mathcal{L}_{g}(E_i)\neq \emptyset$ or $\mathcal{R}_{h}(E_j)\neq \emptyset$. 
					 By Proposition \ref{2-IF and 2-AIF} and (a), we know that $\gl \Lambda\neq 2$ if and only if $\gl \Lambda= 4$. This is equivalent to that  $\Ext_{\Lambda}^{4}(D\Lambda,\Lambda)\neq 0$ holds by \cite[Observation 2.7]{HIMO}. On the other hand, we have
					\begin{align*}
						\Ext^{4}_{\Lambda}(D\Lambda,\Lambda)
						&\simeq\Hom_{\DDD^{\bo}(\mod \Lambda)}(\Lambda,\nu_2^{-1}(\Lambda)[2])\\
						&\simeq\Ext^2_{}(T,T(-\w))\ \ \ \ \ \ \ \ \ \ \ \ \  \ \ \ \  (\text{by (\ref{derived equivalence})} )\\
						&\simeq D\Hom_{}(T,T(2\w)).\ \ \ \ \ \ \ \ \ \ \ \  \ \,   (\text{by (\ref{Auslander-Reiten-Serre duality})} )
					\end{align*}
					This is nonzero since there exist $\x,\y\in \mathcal{L}_{g}(E_i)\cup \mathcal{R}_{h}(E_j)$ such that $\y-\x = -2\w$ and thus  $\gl \Lambda\neq 2$, a contradiction to our assumption. 	
				\end{proof}
					
			\begin{example} Let $\X$ be a GL projective plane of type $(2,2,2,4)$. As in Example \ref{example 2}, the object $T$ given in (\ref{Texample 2}) is  tilting bundle on $\X$ by Theorem \ref{main theorem}. 
				Here is the quiver of the endomorphism algebra $\End (T)$.	
			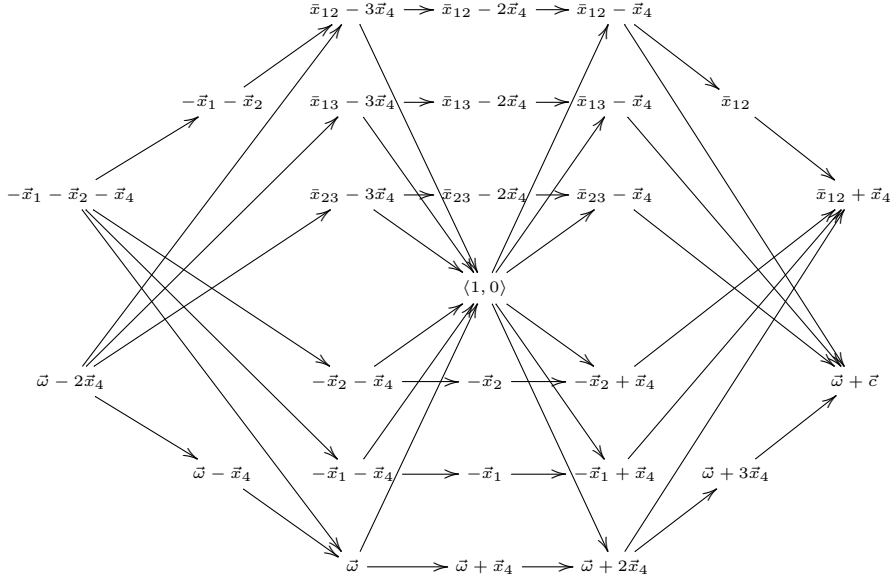
\begin{figure}[H]
				{\tiny\xymatrix@R=3.2em@C=1.5em{						
					&&\bar{x}_{12}-3\x_4 \ar[r] \ar[rddd] & \bar{x}_{12}-2\x_4 \ar[r] & \bar{x}_{12}-\x_4 \ar[rrdddd] \ar[rd] \\
					& -\x_1-\x_2 \ar[ru] & \bar{x}_{13}-3\x_4 \ar[rdd] \ar[r] & \bar{x}_{13}-2\x_4 \ar[r] & \bar{x}_{13}-\x_4 \ar[rrddd] & \bar{x}_{12}\ar[rd] \\
					-\x_1-\x_2-\x_4 \ar[rrdd] \ar[rrddd] \ar[rrdddd]\ar[ru]&& \bar{x}_{23}-3\x_4 \ar[rd] \ar[r] & \bar{x}_{23}-2\x_4 \ar[r] & \bar{x}_{23}-\x_4 \ar[rrdd] &&\bar{x}_{12}+\x_4 \\
					&&&\langle 1,0 \rangle \ar[ruuu] \ar[ruu] \ar[ru] \ar[rddd] \ar[rdd] \ar[rd] &&& \\
					\w-2\x_4 \ar[rruu] \ar[rruuu] \ar[rruuuu]\ar[rd] &&-\x_2-\x_4 \ar[ru] \ar[r] & -\x_2 \ar[r] & -\x_2+\x_4 \ar[rruu] & & \w+\c \\
					& \w-\x_4 \ar[rd] &-\x_1-\x_4 \ar[ruu]\ar[r] &-\x_1 \ar[r] &-\x_1+\x_4 \ar[rruuu] & \w+3\x_4 \ar[ru] \\
					&&\w \ar[ruuu]\ar[r] &\w+\x_4 \ar[r] &\w+2\x_4 \ar[rruuuu] \ar[ru] \\
					}} \caption{The endomorphism algebra of $T$ given in (\ref{Texample 2})} 
				\end{figure}
			\end{example} 
				
			\noindent {\bf Acknowledgements.} Jianmin Chen and Weikang Weng were partially supported by the National Natural Science Foundation of China (Nos. 12371040 and 12131018). Shiquan Ruan was partially supported by Fujian Provincial Natural Science Foundation of China (No. 2024J010006) and the National Natural Science Foundation of China (No. 12271448).

				\vskip 5pt
			\noindent {\scriptsize   \noindent Jianmin Chen, Shiquan Ruan and Weikang Weng\\
				School of Mathematical Sciences, \\
				Xiamen University, Xiamen, 361005, Fujian, PR China.\\
				E-mails: chenjianmin@xmu.edu.cn, sqruan@xmu.edu.cn,
				wkweng@stu.xmu.edu.cn\\ }
			\vskip 3pt
				
			\end{document}